\documentclass[11pt, reqno, psamsfonts]{amsart}

\pdfoutput=1 

\usepackage[normalem]{ulem}
\usepackage{amssymb}
\usepackage{tkz-graph}
\usepackage{amsthm}
\usepackage{amsmath}
\usepackage{latexsym}
\usepackage[T1]{fontenc}
\usepackage[utf8]{inputenc}
\usepackage[russian, french, english]{babel}

\usepackage{graphicx}
\usepackage{wrapfig}
\usepackage[justification=centering, labelfont=bf]{caption}
\usepackage{mathtools}
\usepackage{tikz}
\usepackage{tkz-euclide}
\usepackage{subcaption}
\usetikzlibrary{calc,patterns,angles,quotes}
\usetikzlibrary{decorations.pathreplacing,angles,quotes}
\usepackage{amsbsy}
\usepackage[inline]{enumitem}
\usepackage{mathrsfs}
\usetikzlibrary{shapes,snakes}
\usetikzlibrary{arrows.meta}
\usetikzlibrary{decorations.pathmorphing}
\usetikzlibrary{patterns}
\usepackage{float}
\usepackage{array}
\usepackage{multicol}
\usepackage{stmaryrd}
\usepackage{cancel} 
\usepackage{lmodern}
\usepackage{upgreek}
\usepackage{titlesec}
\usepackage{bbm}
\usepackage[spacing=true,kerning=true,babel=true,tracking=true]{microtype}
\usepackage[shortcuts]{extdash}
\usepackage[foot]{amsaddr}
\usepackage[left=1in,right=1in,top=0.9in,bottom=0.9in,bindingoffset=0cm]{geometry}
\usepackage{bm}
\usepackage{centernot}
\usepackage{mdframed}
\usepackage{multirow}
\usepackage{hyperref}
\hypersetup{
    colorlinks=true,
    linkcolor=blue,
    filecolor=magenta,
    urlcolor=red,
    citecolor=cyan,
}
\allowdisplaybreaks

\usepackage[backend=biber, style=alphabetic, sorting=nyt, maxnames=100,backref=true, sortcites = true]{biblatex}
\title{Borel Line Graphs}
\date{}
\author{\lsstyle James~Anderson}
\email{james.anderson@math.gatech.edu}
\author{\lsstyle Anton~Bernshteyn}
\email{bernshteyn@math.ucla.edu}
\address{\normalfont{(JA) School of Mathematics, Georgia Institute of Technology, Atlanta, GA, USA}}
\address{\normalfont{(AB) Department of Mathematics, University of California, Los Angeles, CA, USA}}
\thanks{Research is partially supported by the NSF grant DMS-2045412 and the NSF CAREER grant DMS-2239187.}

\newtheoremstyle{bfnote}%
{}{}%
{\slshape}{}%
{\bfseries}{\bfseries.}%
{ }%
{\thmname{#1}\thmnumber{ #2}\thmnote{ \ep{\normalfont{}#3}}}

\theoremstyle{bfnote}
\newtheorem{Theorem}{Theorem}[section]
\newtheorem*{Theorem*}{Theorem}

\newtheorem{Lemma}[Theorem]{Lemma}
\newtheorem*{Lemma*}{Lemma}
\newtheorem{Claim}[Theorem]{Claim}

\newtheorem{Corollary}[Theorem]{Corollary}

\newtheorem*{Corollary*}{Corollary}

\theoremstyle{definition}
\newtheorem{Definition}[Theorem]{Definition}
\newtheorem*{Definition*}{Definition}
\newtheorem{Example}[Theorem]{Example}

\newtheorem*{Example*}{Example}

\theoremstyle{remark}
\newtheorem*{ques*}{Question}
\newtheorem*{remk*}{Remark}

\makeatletter
\newcommand{\neutralize}[1]{\expandafter\let\csname c@#1\endcsname\count@}
\makeatother

\newcommand*{\myproofname}{Proof}
\newenvironment{claimproof}[1][\myproofname]{\begin{proof}[#1]}{\end{proof}}

\newcommand{\0}{\emptyset}
\newcommand{\set}[1]{\{#1\}}
\newcommand{\N}{{\mathbb{N}}}
\newcommand{\Z}{\mathbb{Z}}

\newcommand{\R}{\mathbb{R}}

\newcommand{\E}{\mathbb{E}}
\renewcommand{\epsilon}{\varepsilon}

\renewcommand{\phi}{\varphi}
\renewcommand{\theta}{\vartheta}
\renewcommand{\leq}{\leqslant}
\renewcommand{\geq}{\geqslant}
\newcommand{\defeq}{\coloneqq}
\newcommand{\im}{\mathrm{im}}
\newcommand{\bemph}[1]{{\normalfont#1}} 
\newcommand{\ep}[1]{\bemph{(}#1\bemph{)}} 

\newcommand{\pto}{\dashrightarrow}
\newcommand{\emphdef}[1]{\textbf{\textit{{#1}}}}

\numberwithin{equation}{section}
\newcommand{\emphd}[1]{\emphdef{#1}}

\newcommand{\K}{\mathbb{K}}

\newif\ifcomment
\commenttrue   

\titleformat{\section}[block]{\scshape}{\thesection.}{1ex}{}
\titleformat{\subsection}[block]{\bfseries}{\thesubsection.}{1ex}{}
\titleformat{\subsubsection}[runin]{\itshape}{\bfseries\upshape\thesubsubsection.}{1ex}{}[.---]

\titlespacing*{\section}{0pt}{*3}{*1}
\titlespacing*{\subsection}{0pt}{*3}{*1}
\titlespacing*{\subsubsection}{0pt}{*1.5}{*0}

\renewbibmacro{in:}{}

\renewbibmacro*{volume+number+eid}{%
	\printfield{volume}%
	\setunit*{\addnbspace}
	\printfield{number}%
	\setunit{\addcomma\space}%
	\printfield{eid}}

\DeclareFieldFormat[article]{volume}{\textbf{#1}\space}
\DeclareFieldFormat[article]{number}{\mkbibparens{#1}}

\DeclareFieldFormat{journaltitle}{#1,}
\DeclareFieldFormat[thesis]{title}{\mkbibemph{#1}\addperiod}
\DeclareFieldFormat[article, unpublished, thesis]{title}{\mkbibemph{#1},}
\DeclareFieldFormat[book]{title}{\mkbibemph{#1}\addperiod}
\DeclareFieldFormat[unpublished]{howpublished}{#1, }

\DeclareFieldFormat{pages}{#1}

\DeclareFieldFormat[article]{series}{#1\addcomma}

\setlist{topsep=3pt,itemsep=3pt}

\newsavebox\ideabox
\newenvironment{idea}
  {\begin{equation}
   \begin{lrbox}{\ideabox}
   \begin{minipage}{\dimexpr\columnwidth-2\leftmargini}
   \setlength{\leftmargini}{0pt}%
   \begin{quote}}
  {\end{quote}
   \end{minipage}
   \end{lrbox}\makebox[0pt]{\usebox{\ideabox}}
   \end{equation}}

\begin{document}

\vspace*{-19pt}

\maketitle

\begin{abstract}
    We characterize Borel line graphs in terms of 10 forbidden induced subgraphs, namely the 9 finite graphs from the classical result of Beineke together with a 10th infinite graph associated to the equivalence relation $\E_0$ on the Cantor space. As a corollary, we prove a partial converse to the Feldman--Moore theorem, which allows us to characterize all locally countable Borel line graphs in terms of their Borel chromatic numbers.
\end{abstract}

\section{Introduction}\label{Section: Intro}
For a set $X$ and $k \in \N$, we use $[X]^k$ to denote the set of all $k$-element subsets of $X$. When $A \subseteq X$, we let $A^c  \defeq X \setminus A$ be the complement of $A$. All graphs in this paper are simple, i.e., a graph $G$ consists of a vertex set $V(G)$ and an edge set $E(G) \subseteq [V(G)]^2$. When there is no chance of confusion, we use the standard graph-theoretic convention and write $xy$ instead of $\set{x,y}$ to indicate an edge joining vertices $x$ and $y$. \emphd{The line graph} $L(G)$ of a graph $G$ is defined by:
\begin{align*}
    V(L(G)) \,&\defeq\, E(G),\\
    E(L(G)) \,&\defeq\, \{\{e, e'\} \in [E(G)]^2 \,:\, e \cap e' \neq \0\}.
\end{align*}
We say that a graph $L$ is \emphd{a line graph} if it is isomorphic to the line graph of some graph $G$. Beineke famously characterized all line graphs by a list of 9 forbidden induced subgraphs: 
\begin{Theorem}[{Beineke \cite{Beineke}}]\label{Theorem: Beineke}
    A graph is a line graph if and only if it does not have an induced subgraph isomorphic to any of the 9 graphs in Fig.~\ref{Figure: 9 subgraphs}. 
\end{Theorem}

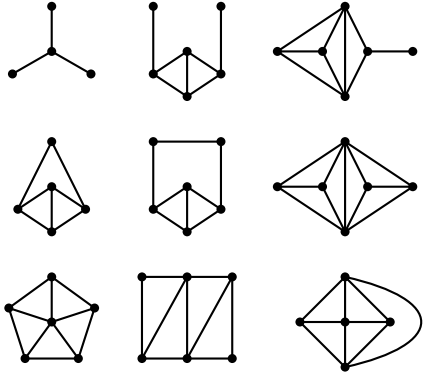
\begin{figure}[t]
\begin{tikzpicture}[
simple/.style={circle, draw = black, fill = black, very thick, inner sep=2pt, minimum size=1mm}, scale=0.6]


\fill[black] (0,0) circle (0.1);
\fill[black] (0,1) circle (0.1);
\fill[black] (-.95, .31) circle (0.1);
\fill[black] (-.59, -.81) circle (0.1);
\fill[black] (.95, .31) circle (0.1);
\fill[black] (.59, -.81) circle (0.1);
\draw[thick] (0, 0) -- (0, 1);
\draw[thick] (0, 0) -- (-.95, .31);
\draw[thick] (0, 0) -- (-.59, -.81);
\draw[thick] (0, 0) -- (.95, .31);
\draw[thick] (0, 0) -- (.59, -.81);
\draw[thick] (0, 1) -- (-.95, .31);
\draw[thick] (-.95, .31) -- (-.59, -.81);
\draw[thick] (-.59, -.81) -- (.59, -.81);
\draw[thick] (.59, -.81) -- (.95, .31);
\draw[thick] (.95, .31) -- (0, 1);

\fill[black] (0,3) circle (0.1);
\fill[black] (0,4) circle (0.1);
\fill[black] (0,3-1) circle (0.1);
\fill[black] (.75,3-.5) circle (0.1);
\fill[black] (-.75,3-.5) circle (0.1);
\draw[thick] (0, 3) -- (0, 2);
\draw[thick] (0, 3) -- (.75, 3-.5);
\draw[thick] (0, 3) -- (-.75, 3-.5);
\draw[thick] (0, 2) -- (.75, 3-.5);
\draw[thick] (0, 2) -- (-.75, 3-.5);
\draw[thick] (0, 4) -- (.75, 3-.5);
\draw[thick] (0, 4) -- (-.75, 3-.5);

\fill[black] (0,6) circle (0.1);
\fill[black] (0,7) circle (0.1);
\fill[black] (-.87,6-.5) circle (0.1);
\fill[black] (.87,6-.5) circle (0.1);
\draw[thick] (0, 6) -- (0, 7);
\draw[thick] (0, 6) -- (-.87, 6-.5);
\draw[thick] (0, 6) -- (.87, 6-.5);

\fill[black] (3,-.81) circle (0.1);
\fill[black] (3,1) circle (0.1);
\fill[black] (2,-.81) circle (0.1);
\fill[black] (2,1) circle (0.1);
\fill[black] (4,-.81) circle (0.1);
\fill[black] (4,1) circle (0.1);
\draw[thick] (3, -.81) -- (3, 1);
\draw[thick] (2, -.81) -- (2, 1);
\draw[thick] (4, -.81) -- (4, 1);
\draw[thick] (2, -.81) -- (3, 1);
\draw[thick] (3, -.81) -- (4, 1);
\draw[thick] (2, -.81) -- (3, -.81);
\draw[thick] (3, -.81) -- (4, -.81);
\draw[thick] (2, 1) -- (3, 1);
\draw[thick] (3, 1) -- (4, 1);

\fill[black] (0+3,3) circle (0.1);
\fill[black] (.75+3,4) circle (0.1);
\fill[black] (-.75+3,4) circle (0.1);
\fill[black] (0+3,3-1) circle (0.1);
\fill[black] (.75+3,3-.5) circle (0.1);
\fill[black] (-.75+3,3-.5) circle (0.1);
\draw[thick] (0+3, 3) -- (0+3, 2);
\draw[thick] (0+3, 3) -- (.75+3, 3-.5);
\draw[thick] (0+3, 3) -- (-.75+3, 3-.5);
\draw[thick] (0+3, 2) -- (.75+3, 3-.5);
\draw[thick] (0+3, 2) -- (-.75+3, 3-.5);
\draw[thick] (.75+3, 4) -- (.75+3, 3-.5);
\draw[thick] (.75+3, 4) -- (-.75+3, 4);
\draw[thick] (-.75+3, 3-.5) -- (-.75+3, 4);

\fill[black] (0+3,3+3) circle (0.1);
\fill[black] (.75+3,4+3) circle (0.1);
\fill[black] (-.75+3,4+3) circle (0.1);
\fill[black] (0+3,3-1+3) circle (0.1);
\fill[black] (.75+3,3-.5+3) circle (0.1);
\fill[black] (-.75+3,3-.5+3) circle (0.1);
\draw[thick] (0+3, 3+3) -- (0+3, 2+3);
\draw[thick] (0+3, 3+3) -- (.75+3, 3-.5+3);
\draw[thick] (0+3, 3+3) -- (-.75+3, 3-.5+3);
\draw[thick] (0+3, 2+3) -- (.75+3, 3-.5+3);
\draw[thick] (0+3, 2+3) -- (-.75+3, 3-.5+3);
\draw[thick] (.75+3, 4+3) -- (.75+3, 3-.5+3);
\draw[thick] (-.75+3, 3-.5+3) -- (-.75+3, 4+3);

\fill[black] (6.5,0) circle (0.1);
\fill[black] (7.5,0) circle (0.1);
\fill[black] (5.5,0) circle (0.1);
\fill[black] (6.5,1) circle (0.1);
\fill[black] (6.5,-1) circle (0.1);
\draw[thick] (6.5, 0) -- (7.5, 0);
\draw[thick] (6.5, 0) -- (5.5, 0);
\draw[thick] (6.5, 0) -- (6.5, 1);
\draw[thick] (6.5, 0) -- (6.5, -1);
\draw[thick] (7.5, 0) -- (6.5, -1);
\draw[thick] (6.5, -1) -- (5.5, 0);
\draw[thick] (5.5, 0) -- (6.5, 1);
\draw[thick] (6.5, 1) -- (7.5, 0);
\draw[thick] (6.5,1) .. controls (8.75,.5) and (8.75,-.5).. (6.5,-1);

\fill[black] (6,3) circle (0.1);
\fill[black] (5,3) circle (0.1);
\fill[black] (6.5,4) circle (0.1);
\fill[black] (6.5,2) circle (0.1);
\fill[black] (7,3) circle (0.1);
\fill[black] (8,3) circle (0.1);
\draw[thick] (6, 3) -- (5, 3);
\draw[thick] (6, 3) -- (6.5, 4);
\draw[thick] (6, 3) -- (6.5, 2);
\draw[thick] (7, 3) -- (6.5, 4);
\draw[thick] (7, 3) -- (6.5, 2);
\draw[thick] (7, 3) -- (8, 3);
\draw[thick] (5, 3) -- (6.5, 4);
\draw[thick] (5, 3) -- (6.5, 2);
\draw[thick] (6.5, 4) -- (6.5, 2);
\draw[thick] (8, 3) -- (6.5, 4);
\draw[thick] (8, 3) -- (6.5, 2);

\fill[black] (6,3+3) circle (0.1);
\fill[black] (5,3+3) circle (0.1);
\fill[black] (6.5,4+3) circle (0.1);
\fill[black] (6.5,2+3) circle (0.1);
\fill[black] (7,3+3) circle (0.1);
\fill[black] (8,3+3) circle (0.1);
\draw[thick] (6, 3+3) -- (5, 3+3);
\draw[thick] (6, 3+3) -- (6.5, 4+3);
\draw[thick] (6, 3+3) -- (6.5, 2+3);
\draw[thick] (7, 3+3) -- (6.5, 4+3);
\draw[thick] (7, 3+3) -- (6.5, 2+3);
\draw[thick] (7, 3+3) -- (8, 3+3);
\draw[thick] (5, 3+3) -- (6.5, 4+3);
\draw[thick] (5, 3+3) -- (6.5, 2+3);
\draw[thick] (6.5, 4+3) -- (6.5, 2+3);

\end{tikzpicture}
\caption{The 9 graphs of Beineke.}
\label{Figure: 9 subgraphs}
\end{figure}
We are interested in extending Beineke's result to the setting of \emphd{Borel graphs}, that is, graphs $G$ such that $V(G)$ is a standard Borel space and $E(G)$ is a Borel subset of $[V(G)]^2$. (We refer the reader unfamiliar with such terminology to Kechris's book on descriptive set theory \cite{kechris2012classicalDescriptiveSetTheory}; we also review some necessary descriptive set-theoretic background in \S\ref{Section: prelims}.) The systematic study of Borel graphs and their combinatorics was initiated in the landmark 1999 paper by Kechris, Solecki, and Todorcevic \cite{KST1999BorelChromatic}, who applied descriptive set theory to the study of graph colorings. This launched the development of the highly fruitful field of \textit{descriptive combinatorics}, which has connections to many areas of mathematics, including group theory, measure theory, ergodic theory, theoretical computer science, and more. For an overview of the field, see the 2020 survey by Kechris and Marks \cite{KechrisAndMarks2020SurveyDescriptiveGraphCombo} and the 2021 survey by Pikhurko \cite{Pikhurko2021Survey}.

Borel graphs $G$ and $H$ are \emphd{Borel isomorphic}, in symbols $G \cong_B H$, if there exists a Borel isomorphism from $G$ to $H$ (i.e., a graph isomorphism $f \colon V(G) \to V(H)$ that is a Borel function). Note that the line graph $L(G)$ of a Borel graph $G$ is itself a Borel graph. We say that a Borel graph $L$ is \emphd{a Borel line graph} if there exists a Borel graph $G$ such that $L \cong_B{L(G)}$. Clearly, if a graph is a Borel line graph, then it is both a Borel graph and a line graph.
Conversely, we ask:
    \begin{quote}
        \textsl{Given a Borel graph that is a line graph, when is it a Borel line graph?}
    \end{quote}
    
    To demonstrate that this question is nontrivial, let us give an example of a Borel graph that is a line graph but \textit{not} a Borel line graph. Recall that the \emphdef{Cantor space} is the set $\mathcal{C} \defeq \{0,1\}^{\N}$ of countably infinite binary strings endowed with the product topology, where the topology on $\set{0,1}$ is discrete. The equivalence relation $\E_0$ is defined on $\mathcal{C}$ by relating two elements if they are equal after some index; that is, given $\alpha$, $\beta \in \mathcal{C}$, we have
\[\alpha\, \mathbb{E}_0\, \beta  \iff \exists\, m \in \N,\, \forall n\geq m\, (\alpha(n) = \beta(n)).\]
This is a Borel equivalence relation\footnote{As usual, we view binary relations as sets of ordered pairs and say that a binary relation $R$ on a standard Borel space $X$ is \emphd{Borel} if it is a Borel subset of $X^2$.} with countable classes (for a survey on such equivalence relations, see the recent manuscript by Kechris \cite{Kechris2019CBER}). 
From $\E_0$, we define a graph $\K_0$ as follows:
\[V(\K_0) \,\defeq\, \mathcal{C}, \quad \quad E(\K_0) \,\defeq\, \{ \{\alpha, \beta\} \in [\mathcal{C}]^2 : \alpha \, \mathbb{E}_0 \, \beta\}.\]
In other words, $\K_0$ is obtained by making every $\mathbb{E}_0$-equivalence class into a clique. 
Then $\K_0$ is a Borel graph, and, being a collection of vertex-disjoint cliques, it is a line graph (a clique is isomorphic to the line graph of a star). However, $\K_0$ is not a Borel line graph. 
This essentially boils down to the fact that the relation $\E_0$ is not {smooth} (see Definition~\ref{Definition: smooth}), as if $\K_0$ were a Borel line graph, the map picking out the center of the star corresponding to each component of $\K_0$ would witness that $\E_0$ is smooth. 
Another simple proof uses Borel chromatic numbers:

\begin{Definition}[Borel chromatic number]
    Given a graph $G$, a \emphd{proper coloring} of $G$ is a function $f : V(G) \to C$, where $C$ is some set, such that for all $xy \in E(G)$, $f(x) \neq f(y)$. The \emphd{chromatic number} of $G$, denoted by $\chi(G)$, is the smallest cardinality of a set $C$ such that $G$ has a proper coloring $f \colon V(G) \to C$. 
For a Borel graph $G$, its \emphd{Borel chromatic number}, $\chi_B(G)$, is the smallest cardinality of a standard Borel space $X$ such that there exists a Borel proper coloring $f \colon V(G) \to X$.
\end{Definition}

A standard Baire category argument proves that $\chi_B(\K_0) > \aleph_0$ \cite[27]{KechrisAndMarks2020SurveyDescriptiveGraphCombo}. On the other hand, if $\K_0$ were a Borel line graph, then by the Feldman--Moore theorem (see Theorem \ref{Theorem: FM} below), we would have $\chi_B(\K_0) \leq \aleph_0$. It follows that $\K_0$ is not a Borel line graph. 

Given a graph $G$ and a subset $U \subseteq V(G)$, we let $G[U]$ denote the subgraph of $G$ induced by the vertices of $U$, i.e., $G[U] \defeq (U, E(G) \cap 
 [U]^2)$. Given graphs $G$ and $H$, we say that $G$ \emphd{contains a copy of} $H$ if $G$ has an induced subgraph isomorphic to $H$, i.e., if there exists a set $U \subseteq V(G)$ such that $G[U] \cong H$. Similarly, if $G$ and $H$ are Borel graphs, we say that $G$ \emphd{contains a Borel copy of} $H$ if there exists a Borel set $U \subseteq V(G)$ such that $G[U] \cong_B H$. 
 It is clear that the property of being a Borel line graph is preserved under taking Borel induced subgraphs. Therefore, for a Borel graph to be a Borel line graph, it must contain no copies of the 9 forbidden subgraphs of Beineke nor a Borel copy of $\K_0$\footnote{\textit{Borel} copy is important here, for there are Borel graphs $G$ whose line graphs contain a copy, but no Borel copy, of $\K_0$. For example, consider the Borel graph with vertex set $\R$ and edge set $\set{\set{x, x+n} \,: x \in [0,1),\, n \in \Z \setminus \set{0}}$. 
This graph has continuum many components isomorphic to countably infinite stars, and thus its line graph $L$ is isomorphic to $\K_0$. However, since $\K_0$ is not a Borel line graph, $L$ cannot contain a Borel copy of $\K_0$.}.
Our main result is that the converse is true, i.e., $\K_0$ is the \textit{only} additional obstruction in the Borel setting.

\begin{Theorem}\label{Theorem: 10 Graph Theorem}
Let $L$ be a Borel graph. Then $L$ is a Borel line graph if and only if it contains neither a copy of any of the 9 graphs of Beineke nor a Borel copy of $\K_0$. 
\end{Theorem}

In view of Theorem~\ref{Theorem: Beineke}, the above statement is equivalent to the assertion that a Borel graph $L$ is a Borel line graph if and only if it is a line graph that does not contain a Borel copy of $\K_0$.

As an immediate consequence of Theorem~\ref{Theorem: 10 Graph Theorem}, we obtain a characterization of locally countable Borel line graphs in terms of their Borel chromatic numbers. This follows from the graph-theoretic version of the Feldman--Moore theorem \cite{FeldmanMoore} due to Kechris, Solecki, and Todorcevic \cite{KST1999BorelChromatic}, which states that locally countable Borel line graphs have countable Borel chromatic numbers.

\begin{Theorem}[{Feldman--Moore: Graph version \cite[Theorem 4.10]{KST1999BorelChromatic}}]\label{Theorem: FM}
If $L$ is a locally countable Borel line graph, then $\chi_B(L) \leq \aleph_0$. 
\end{Theorem}

As mentioned above, $\chi_B({\mathbb{K}_0}) > \aleph_0$, and thus Theorems~\ref{Theorem: Beineke} and \ref{Theorem: 10 Graph Theorem} imply the following partial converse to the Feldman--Moore theorem:

\begin{Corollary}\label{Corollary: converse to Feldman Moore} Let $L$ be a locally countable Borel graph. If $L$ is a line graph and $\chi_B(L) \leq \aleph_0$,
    then $L$ is a Borel line graph.     \end{Corollary}
    
Theorem \ref{Theorem: 10 Graph Theorem} allows the Feldman--Moore theorem to be stated in terms of forbidden subgraphs: 

\begin{Corollary}\label{Corollary: FM version 2}
  If $L$ is a locally countable Borel graph that contains neither a copy of any of the 9 graphs of Beineke nor a Borel copy of $\mathbb{K}_0$, then $\chi_{B}(L) \leq \aleph_0$.   
\end{Corollary}

An intriguing question is whether the hypotheses of Corollary \ref{Corollary: FM version 2} can be weakened. For instance, it would be interesting to know if forbidding only some of the 9 graphs of Beineke together with $\K_0$ is enough to reach the same conclusion. This line of inquiry can be naturally viewed as an extension to the Borel setting of the theory of $\chi$-boundedness \cite{chiBoundedSurvey}, which aims to bound the chromatic number of a graph with certain forbidden substructures by a function of its clique number. More broadly, our work indicates the prospect of fruitful interactions between descriptive set theory and \emph{structural} (as opposed to extremal or probabilistic) graph theory and leads to general problems such as what other natural classes of Borel graphs can be characterized by means of excluding certain Borel induced subgraphs.

The remainder of the paper is organized as follows. In \S\ref{Section: prelims}, we review some necessary background and tools from descriptive set theory. A reader familiar with descriptive set theory may proceed directly to \S\ref{section: proof outline}, where we provide a road map for the proof of Theorem \ref{Theorem: 10 Graph Theorem}. The rest of the paper contains  proofs of the intermediate theorems and lemmas needed in the proof of Theorem \ref{Theorem: 10 Graph Theorem}.

\section{Tools from descriptive set theory}\label{Section: prelims} 

In this section, we provide some fundamental tools from descriptive set theory and the study of Borel equivalence relations. Our main references for descriptive set theory are \cite{kechris2012classicalDescriptiveSetTheory,AnushDST}, and the reader is invited to consult them for any background not mentioned here.

A \emphd{standard Borel space} is a set $X$ equipped with a $\sigma$-algebra $\mathfrak{B}(X)$ (called \emphd{Borel subsets}) that coincides with the Borel $\sigma$-algebra generated by some Polish (i.e., separable completely metrizable) topology on $X$. All uncountable standard Borel spaces are isomorphic \cite[Theorem 15.6]{kechris2012classicalDescriptiveSetTheory}, so there is usually no loss of generality in assuming that $X$ is some specific space such as $\R$ or the Cantor space $\mathcal{C}$. If $X$ is a standard Borel space and $A \subseteq X$ is a Borel set, then $A$ equipped with the $\sigma$-algebra $\set{B \cap A \,:\, B \in \mathfrak{B}(X)}$ is also a standard Borel space \cite[Corollary 13.4]{kechris2012classicalDescriptiveSetTheory}. A function $f \colon X \to Y$ between two standard Borel spaces is \emphd{Borel} if for every Borel set $A \subseteq Y$, its preimage $f^{-1}(A)$ is a Borel subset of $X$. Equivalently, a function $f \colon X \to Y$ is Borel if and only if $\mathsf{graph}(f) \defeq \set{(x,y) \,:\, f(x) = y}$ is a Borel subset of $X \times Y$ \cite[Theorem 14.12]{kechris2012classicalDescriptiveSetTheory}.

It is often convenient to describe a subset $B \subseteq X$ via a statement 
        $P(x)$ with one free variable $x$ such that $B = \set{x \in X \,:\, P(x)}$. To verify such a set $B$ is Borel, we will usually not explicitly write its definition out set-theoretically, but instead rely on the form of the statement $P(x)$ itself, keeping in mind that conjunctions and universal quantifiers (resp.~disjunctions and existential quantifiers) in $P(x)$ correspond to intersections (resp.~unions) in the construction of $B$, while negations correspond to complements. 

    Let $X$ be a standard Borel space. The \emphd{diagonal} of $X$ is the set \[\Delta(X) \,\defeq\, \set{(x,x) \,:\, x \in X} \,\subseteq\, X^2.\] Note that $\Delta(X)$ is a Borel subset of $X^2$ (it is the graph of the identity function on $X$). Now consider the map $\mathsf{pair} \colon X^2 \setminus \Delta(X) \to [X]^2$ given by $\mathsf{pair}(x,y) \defeq \set{x,y}$. We endow $[X]^2$ with the $\sigma$-algebra
    \[
        \mathfrak{B}([X]^2) \,\defeq\, \set{A \subseteq [X]^2 \,:\, \mathsf{pair}^{-1}(A) \in \mathfrak{B}(X^2)}.
    \]
    This makes $[X]^2$ a standard Borel space \cite[Example 6.1 and Proposition 6.3]{KechrisMiller}. By construction, the function $\mathsf{pair} \colon X^2 \setminus \Delta(X) \to [X]^2$ is Borel.


\begin{Definition}[Analytic and coanalytic sets]
    Let $X$ be a standard Borel space. A set $A \subseteq X$ is \emphd{analytic} if there exist a standard Borel space $Y$ and a Borel function $f : Y \to X$ such that $f(Y) = A$. Equivalently, $A$ is analytic if there exist a standard Borel space $Y$ and a Borel set $B \subseteq X \times Y$ such that $x \in A \iff \exists\, y \in Y\, ((x,y) \in B)$. A set $A \subseteq X$ is \emphd{coanalytic} if its complement is analytic.
\end{Definition}

In practice, to show a set $A$ is analytic we will typically write 
\[x \in A \iff \exists\, y \in Y\, (P(x,y)),\]
where $P(x,y)$ is a statement with two free variables such that $P(x,y)$ holds if and only if $(x,y) \in B$ for some Borel set $B \subseteq X \times Y$. Verifying that 
$P(x,y)$ really does correspond to a Borel set will often be routine and left to the reader. Similarly, a set $A \subseteq X$ is coanalytic when
\[x \in A \iff \forall\, y \in Y\, (P(x,y)),\]
where $P(x,y)$ describes a Borel subset of $X \times Y$. We give an example of a proof of this type below; in the sequel, equally straightforward arguments will be omitted.
\begin{Example}\label{Example: Analytic1}
    For a Borel graph $G$, the set $I \subseteq V(G)$ of all isolated vertices is coanalytic. Indeed,
    \[
        x \in I  \iff   \forall\, y \in V(G)\, (xy \notin E(G)).
    \]
    To see that the set $\set{(x,y) \in V(G)^2 \,:\, xy \notin E(G)}$ is Borel, we observe that it is equal to
    \[
        V(G)^2 \setminus \mathsf{pair}^{-1}(E(G)).
    \]
    It follows that the set $I$ is coanalytic.
\end{Example}

    The family of all analytic subsets of a standard Borel space is closed under countable unions and intersections, and the same is true for the family of all coanalytic subsets \cite[Proposition 14.4]{kechris2012classicalDescriptiveSetTheory}.

\begin{Example}\label{Example: Analytic2}
    We argue that for a Borel graph $G$, the equivalence relation $\equiv_G$ whose classes are the connected components of $G$ is analytic. Indeed, we have
    \begin{align*}
        x \equiv_G y  \iff \exists\, d \in \N, \, \exists\, &(u_0, \ldots, u_d) \in V(G)^{d+1}\\
        &\big(x = u_0, \, u_0u_1 \in E(G), \, \ldots, \, u_{d-1}u_d \in E(G),\, u_d = y\big).
    \end{align*}
    This means that we can write ${\equiv_G} = \bigcup_{d \in \N} S_d$, where
    \begin{align*}
        (x,y) \in S_d  \iff  \exists\, &(u_0, \ldots, u_d) \in V(G)^{d+1}\\
        &\big(x = u_0, \, u_0u_1 \in E(G), \, \ldots, \, u_{d-1}u_d \in E(G),\, u_d = y\big).
    \end{align*}
    Since $E(G)$ is Borel, we see that each set $S_d$ is analytic, and hence their union is analytic as well.
\end{Example}

    It should be noted that there exist analytic sets that are not Borel. This follows from a diagnolization argument originally given by Suslin, see \cite[Theorem 14.2]{kechris2012classicalDescriptiveSetTheory}. We now present a classical result of Luzin and Novikov that provides a sufficient condition for an analytic set to be Borel. A proof can be found in \cite[Theorem 18.10]{kechris2012classicalDescriptiveSetTheory}
or \cite[Theorem 32]{miller2012graphTheoreticApproachToDescriptiveSetTheory}.

\begin{Theorem}[Luzin--Novikov]\label{Theorem: Luzin--Novikov} Let $X$ and $Y$ be standard Borel spaces and let $B \subseteq X \times Y$ be a Borel set. If for every $x \in X$, the set
$\{y \in Y : (x,y) \in B\}$ is countable, then 
\[\{x \in X \,:\, \exists\, y \in Y \, ((x,y) \in B)\}\]
is a Borel subset of $X$.
\end{Theorem}


\begin{Example}\label{Example: Luzin--Novikov}
    Thanks to  the Luzin--Novikov theorem, many combinatorial constructions on locally countable Borel graphs can be shown to result in Borel sets. For instance, as in Example~\ref{Example: Analytic1}, let $I \subseteq V(G)$ be the set of all isolated vertices of $G$. Then
    \[
        x \in I^c  \iff   \exists\, y \in V(G)\, (xy \in E(G)).
    \]
    If $G$ is a locally countable Borel graph, then for each $x \in V(G)$, the set $\set{y \in V(G) \,:\, xy \in E(G)}$ is countable, and hence, by the Luzin--Novikov theorem, the set $I^c$ is Borel. It follows that the set of all isolated vertices in a locally countable Borel graph is Borel. A similar argument shows that for a locally countable Borel graph $G$, the relation $\equiv_G$ defined in Example~\ref{Example: Analytic2} is Borel.
\end{Example}

Another classical theorem of Suslin allows separating two analytic sets by a Borel set.  A proof can be found in \cite[Theorem 14.7]{kechris2012classicalDescriptiveSetTheory}.     
\begin{Theorem}[Analytic separation] \label{Theorem: Analytic Separation}
Let $A_1$ and $A_2$ be disjoint analytic subsets of a standard Borel space $X$. Then there exists a Borel set $B \subseteq X$ such that $A_1 \subseteq B \subseteq A_2^c$.
\end{Theorem}

An immediate corollary of Theorem~\ref{Theorem: Analytic Separation} is that a set that is both analytic and coanalytic must be Borel, since it can be separated from its complement by a Borel set.

Given a set $X$ and an equivalence relation $E$ on $X$, we say a set $A \subseteq X$ is \emphd{$E$-invariant} if no element of $A$ is related by $E$
 to an element of $A^c$. The following result is 
 \cite[Exercise 14.14]{kechris2012classicalDescriptiveSetTheory}; we provide a proof for completeness.

\begin{Lemma}[Invariant analytic separation]\label{Lemma: invariant analytic separation}
Let $X$ be a standard Borel space and let $E$ be an analytic equivalence relation on $X$. 
Suppose $Y$, $Z \subseteq X$ are analytic sets such that no element of $Y$ is $E$-related to an element of $Z$.
Then there is an $E$-invariant Borel set $B \subseteq X$ such that $Y \subseteq B \subseteq Z^c$. 
\end{Lemma}

\begin{proof}
    Given $A \subseteq X$, let $[A]_E$ be the \emphd{$E$-saturation} of $A$, i.e.,
\[[A]_E \,\defeq\, \{x \in X \,\colon\, \exists\, a \in A\, (x\, E\, a)\}.\]
Observe that the set $[A]_E$ is $E$-invariant; furthermore, since $E$ is analytic, if $A$ is analytic, then $[A]_E$ is analytic as well. Upon replacing $Y$ by $[Y]_E$ and $Z$ by $[Z]_E$, we may assume that $Y$ and $Z$ are $E$-invariant and disjoint.

We will now inductively define an increasing sequence of Borel sets $(B_i)_{i \in \N}$ such that $Y \subseteq B_i \subseteq Z^c$ and $[B_i]_E \subseteq B_{i+1}$.
We do so as follows: by analytic separation, there exists a Borel set $B_0$ such that $Y \subseteq B_0 \subseteq Z^c$. Now let $B_i$ be Borel with $Y \subseteq B_i \subseteq Z^c$. Since $B$ is Borel, it follows $[B_i]_E$ is analytic; furthermore, as $B_i \subseteq Z^c$ and $Z$ is $E$-invariant, it follows $[B_i]_E \subseteq Z^c$. Thus by analytic separation there exists Borel $B_{i+1}$ with $[B_i]_E \subseteq B_{i+1} \subseteq Z^c$. This completes the inductive construction.

Let $B \defeq \bigcup_{n \in \N}B_n$. Clearly $B$ is Borel and $Y \subseteq B \subseteq Z^c$. Finally, since $B = \bigcup_{n \in \N} [B_n]_E$, the set $B$ is $E$-invariant, as desired. 
\end{proof}

An important role in our arguments will be played by the following special class of equivalence relations on standard Borel spaces:

\begin{Definition}[Smoothness]\label{Definition: smooth}
    Let $E$ be an equivalence relation on a standard Borel space $X$. Then $E$ is \emphdef{smooth} if there exist a standard Borel space $Y$ and Borel function $f : X \to Y$ such that for all $x$, $y \in X$ we have $x\, E \, y \iff f(x) = f(y)$. We say $f$ \emphd{witnesses the smoothness} of $E$.
\end{Definition}

Note that a smooth equivalence relation on a standard Borel space is automatically Borel. Since all uncountable standard Borel spaces are isomorphic \cite[Theorem 15.6]{kechris2012classicalDescriptiveSetTheory}, we may, without loss of generality, use $\R$ as the codomain of $f$ in Definition~\ref{Definition: smooth} for concreteness.

Recall the equivalence relation $\E_0$ discussed in \S\ref{Section: Intro}. Harrington, Kechris, and Louveau showed that not only is $\E_0$ nonsmooth, but it is a ``smallest'' nonsmooth equivalence relation \cite{HarringtonKechrisLouveau1990Glimm-EffrosE0Dichotomy}. This result is known as the $\E_0$-dichotomy. While the original proof due to Harrington, Kechris, and Louveau uses methods of effective descriptive set theory, a classical, graph-theoretic proof was given by Miller \cite[Theorem 26]{miller2012graphTheoreticApproachToDescriptiveSetTheory}. 



\begin{Theorem}[$\E_0$-dichotomy]\label{Theorem: E0 dichotomy}
   Let $X$ be a standard Borel space and let $E$ be a Borel equivalence relation on $X$. Then exactly one of the following holds:
    \begin{enumerate}[label=\ep{\normalfont\arabic*}]
        \item $E$ is smooth, or
        \item there exists a Borel embedding from $\E_0$ to $E$; that is, there is an injective Borel function $f : \mathcal{C} \to X$ such that $\alpha\, \E_0 \, \beta \iff f(\alpha)\, E\, f(\beta)$.
    \end{enumerate}
\end{Theorem}



\section{Outline of the proof of Theorem \ref{Theorem: 10 Graph Theorem}}\label{section: proof outline}

\subsection{Line graph decompositions and line graph relations}

 An important role in our arguments in played by a characterization of line graphs via a partition of their edges into cliques, which we call a {line graph decomposition}. 

\begin{Definition}[Line graph decompositions]\label{Def: line graph decomposition}
A \emphdef{line graph decomposition} of a graph $L$ is a collection $\mathscr{C}$ of nonempty subsets of $V(L)$ such that:
\begin{itemize}
    \item
        for all $C \in \mathscr{C}$, $L[C]$ is a clique;
    \item
        the sets $E(L[C])$, $C \in \mathscr{C}$ are pairwise disjoint and their union is $E(L)$; 
    \item each non-isolated vertex of $L$ is contained in exactly two sets in $\mathscr{C}$; each isolated vertex is contained in exactly one set in $\mathscr{C}$.
\end{itemize}
\end{Definition}

Note that any two sets in a line graph decomposition have at most one common vertex. Krausz observed that a graph $L$ is a line graph if and only if $L$ has a line graph decomposition \cite{krausz1943}. Indeed, 
if $L = L(G)$ for some graph $G$, then $\set{\set{e \in E(G) \,:\, e \ni v} \,:\, v \in V(G)}$ is a line graph decomposition of $L$. Conversely, given a line graph decomposition $\mathscr{C}$ of a graph $L$ without isolated vertices, one can form a graph $G$ such that $L(G) \cong L$ as follows:
    \begin{align}
        V(G) \,\defeq\, \mathscr{C},\quad\quad \label{Eq: line graph decomposition}
        E(G) \,\defeq\, \{\{C, C'\} \in [\mathscr{C}]^2 : C \cap C' \neq \0\}.
    \end{align}
    Here an isomorphism $\phi: E(G) \to V(L)$ is given by letting $\phi(\{C, C'\})$ for each edge $\set{C,C'} \in E(G)$ be the unique vertex in $C \cap C'$.
If $L$ has isolated vertices, we simply add to $G$ an isolated edge corresponding to each isolated vertex of $L$. 

Since a line graph decomposition of $L$ induces a partition of the edge set of $L$, we can use it to define an equivalence relation on $E(L)$, called the line graph relation:

\begin{Definition}[Line graph relations]
Let $\mathscr{C}$ be a line graph decomposition of a graph $L$. The equivalence relation
\[{\sim_{\mathscr{C}}} \,\defeq\, \{(e, e') \in E(L)^2 \,:\, \exists\, C \in \mathscr{C}\, \big(\{e, e'\} \subseteq E(L[C])\big)\}\]
on $E(L)$ is called \emphdef{the line graph relation} on $L$ under $\mathscr{C}$. An equivalence relation $\sim$ on $E(L)$ is called \emphd{a line graph relation} if there is a line graph decomposition $\mathscr{C}$ of $L$ such that ${\sim} = {\sim_\mathscr{C}}$.
\end{Definition}

Note that if $\sim$ is a line graph relation on $L$, then each $\sim$-class is the edge set of a clique in $L$. Furthermore, the following combinatorial characterization of line graph relations is an immediate consequence of the above definitions:

\begin{Lemma}[{\cite[130]{Beineke}}]\label{lemma: equiv relation is line graph decomp relation}
    Let $L$ be a graph. An equivalence relation $\sim$ on $E(L)$ is a line graph relation if and only if each $\sim$-equivalence class is the edge set of a clique in $L$ and every vertex of $L$ is incident to at most two $\sim$-classes.
\end{Lemma}

Let $G$ be a graph and let $R$ be a binary relation on $E(G)$. If $H \subseteq G$ is a subgraph of $G$, we let $R|_{H} \defeq R \cap E(H)^2$ be the \emphdef{restriction} of $R$ to $H$. If $U \subseteq V(G)$, then we let $R|_{U} \defeq R|_{G[U]}$ be the \emphdef{restriction} of $R$ to $U$.

\begin{Lemma}\label{lemma: line graph relation reduction}
    Let $L$ be a graph with a line graph relation $\sim$. If 
    $H$ is an induced subgraph of $L$, then ${\sim}|_{H}$ is a line graph relation on $H$. 
\end{Lemma}
\begin{proof}
    Follows immediately from Lemma \ref{lemma: equiv relation is line graph decomp relation}. 
\end{proof}

\subsection{Main steps of the proof of Theorem~\ref{Theorem: 10 Graph Theorem}}

We can now describe the key steps in the proof of our main result. To begin with, we note that, thanks to Beineke's Theorem~\ref{Theorem: Beineke}, Theorem~\ref{Theorem: 10 Graph Theorem} is equivalent to the following statement:

\begin{Theorem}\label{Theorem: K0 Dichotomy}
Let $L$ be a Borel graph that is a line graph. Then $L$ is a Borel line graph if and only if $L$ does not contain a Borel copy of
$\mathbb{K}_0$.
\end{Theorem}

A significant complication in proving Theorem~\ref{Theorem: K0 Dichotomy} arises from the fact that $L$ is not assumed to be locally countable. As briefly discussed in Example~\ref{Example: Luzin--Novikov}, in the study of locally countable Borel graphs, the \hyperref[Theorem: Luzin--Novikov]{Luzin--Novikov theorem} is routinely used to show that various combinatorially defined sets are Borel, but such arguments are unavailable for general Borel graphs. As a result, even very simple sets associated to $L$, such as the set of all isolated vertices, may fail to be Borel. This makes analyzing the structure of $L$ through the lens of Borel combinatorics a particularly intricate task. 

 Let $L$ be a Borel graph that is a line graph. Below we outline the major intermediate results that go into the proof of Theorem~\ref{Theorem: K0 Dichotomy}. Each of them presents interesting challenges in its own right, which we will comment on in the subsequent subsections.
 
 Since $L$ is a line graph, it has a line graph decomposition, and hence there is a line graph relation on $L$. The first step in our proof is to find a \emph{Borel} line graph relation on $L$:

 \begin{Theorem}[Borel line graph relations]\label{Theorem: A Borel line graph decomposition Exists} If $L$ is a Borel graph that is a line graph, then $L$ has a Borel line graph relation.
\end{Theorem}
\begin{proof}
\S\ref{Section: Proof of Theorem A Borel line graph decomposition Exisits}.
\end{proof}

    Theorem~\ref{Theorem: A Borel line graph decomposition Exists} yields a Borel line graph relation regardless of whether $L$ is a Borel line graph. The question arises: Given a Borel graph $L$ with a Borel line graph relation $\sim$, how can we tell whether $L$ is a Borel line graph? The answer is given by the following theorem:

    \begin{Theorem}[Smooth line graph relations]\label{Theorem: smoothness implies Borel G} Let $L$ be a Borel graph that is a line graph. Then the following are equivalent:
\begin{enumerate}[label=\ep{\normalfont\arabic*}]
    \item\label{item:Blg} $L$ is a Borel line graph,
    \item\label{item:onesmooth} $L$ has a smooth line graph relation,
    \item\label{item:allsmooth} all Borel line graph relations on $L$ are smooth.
\end{enumerate}
\end{Theorem}

Of the above equivalences, we only use \ref{item:Blg} $\Longleftrightarrow$ \ref{item:onesmooth} in our proof of Theorem~\ref{Theorem: 10 Graph Theorem}. Still, the equivalence \ref{item:onesmooth} $\Longleftrightarrow$ \ref{item:allsmooth} is of independent interest, as it is a natural addition to our characterization of 
Borel line graphs by the smoothness of their Borel line graph relations. 
The implication \ref{item:onesmooth} $\Rightarrow$ \ref{item:allsmooth} is 
proved in Appendix~\ref{appendix one smooth implies all smooth}, and the details of \ref{item:onesmooth} $\Rightarrow$ \ref{item:Blg} are presented in \S\ref{Section: proof of Theorem smoothness implies Borel G} (see also \S\ref{section: Analyzing smooth line graph relations} for an informal discussion of this implication). The other implications in Theorem~\ref{Theorem: smoothness implies Borel G} are straightforward: 

\begin{proof}
    \ref{item:Blg} $\Rightarrow$ \ref{item:onesmooth}: Without loss of generality, we may assume that $L = L(G)$ for some Borel graph $G$. 
    The following definition gives a line graph relation on $E(L)$:
\[ \{e_1, e_2\} \sim \{e_1', e_2'\} \iff e_1 \cap e_2 = e_1' \cap e_2'.\]
The smoothness of $\sim$ is witnessed by the function $f: E(L) \to V(G)$ defined by letting $f(\{e_1, e_2\})$ be the unique vertex in $e_1 \cap e_2$. 
Therefore, $L$ has a smooth line graph relation. 

    \ref{item:onesmooth} $\Rightarrow$ \ref{item:Blg}: \S\ref{Section: proof of Theorem smoothness implies Borel G}.

    \ref{item:onesmooth} $\Rightarrow$ \ref{item:allsmooth}: Appendix \ref{appendix one smooth implies all smooth}.
    
    \ref{item:allsmooth} $\Rightarrow$ \ref{item:onesmooth}: Follows from Theorem \ref{Theorem: A Borel line graph decomposition Exists}.
\end{proof}

    Assuming $L$ is not a Borel line graph, Theorems~\ref{Theorem: A Borel line graph decomposition Exists} and \ref{Theorem: smoothness implies Borel G} yield a nonsmooth Borel line graph relation $\sim$ on $E(L)$, which we utilize in \S\ref{Section: reduction to locally countable case} to find a desired Borel copy of $\K_0$ in $L$. 




    In the following subsections we describe in a little more detail some of the ideas used in accomplishing these steps. 

\subsection{Finding a Borel line graph relation \ep{Theorem~\ref{Theorem: A Borel line graph decomposition Exists}}}


A key ingredient in our proof of Theorem~\ref{Theorem: A Borel line graph decomposition Exists} is the fact that a connected line graph has a \emph{unique} line graph decomposition, save for four graphs. This is shown in Corollary \ref{Corollary: nonsingular clique cover is unique} below. These four graphs are listed in the second column of Table~\ref{tab:exceptional} and illustrated in Fig.~\ref{Figure: Exceptions to Whitney's Theorem}. We call a graph \emphd{singular} if it is isomorphic to any of these 4 graphs, and we call a graph \emphd{exceptional} if its line graph is singular. The exceptional graphs are listed in the first column of Table~\ref{tab:exceptional} (here $K_{1,3}^+$ is $K_{1,3}$ with one additional edge, and $K_4^-$ is $K_4$ with a single edge removed.)

\begin{table}[t]
    \centering
    \begin{tabular}{l|l}
    $G$ & $L(G)$\\
    \hline
    $K_3$ or $K_{1,3}$ & $K_3$\\
    $K_{1,3}^+$ & $K_4^-$ \\
    $K_4^-$ & Square pyramid \\
    $K_4$ & Octrahedron
    \end{tabular}
    \caption{Exceptional graphs in Whitney's strong isomorphism theorem (left) with the corresponding line graphs (right).}
    \label{tab:exceptional}
\end{table}

\begin{figure}[t]
\begin{tikzpicture}[
simple/.style={circle, draw = black, fill = black, very thick, inner sep=2pt, minimum size=1mm}, scale=0.9]

\fill[black] (0,-.16) circle (0.1);
\fill[black] (-.5,-1) circle (0.1);
\fill[black] (.5,-1) circle (0.1);
\draw[thick] (0, -.16) -- (-.5, -1);
\draw[thick] (0, -.16) -- (.5, -1);
\draw[thick] (-.5, -1) -- (.5, -1);

\fill[black] (2.5,0) circle (0.1);
\fill[black] (2.5,-1) circle (0.1);
\fill[black] (3.25,-.5) circle (0.1);
\fill[black] (3-1.25,-.5) circle (0.1);
\draw[thick] (2.5, 0) -- (2.5, -1);
\draw[thick] (2.5, 0) -- (3.25, -.5);
\draw[thick] (2.5, 0) -- (2.5-.75, -.5);
\draw[thick] (2.5, -1) -- (3.25, -.5);
\draw[thick] (2.5, -1) -- (2.5-.75, -.5);

\fill[black] (4.5,-.5-.4) circle (0.1);
\fill[black] (5.5,-.1-.4) circle (0.1);
\fill[black] (5.1,-.9-.4) circle (0.1);
\fill[black] (6.1,-.5-.4) circle (0.1);
\draw[thick] (4.5, -.5-.4) -- (5.5, -.1-.4);
\draw[thick] (5.5, -.1-.4) -- (6.1, -.5-.4);
\draw[thick] (6.1, -.5-.4) -- (5.1, -.9-.4);
\draw[thick] (5.1, -.9-.4) -- (4.5, -.5-.4);

\fill[black] (5.3,.5-.4) circle (0.1);
\draw[thick] (5.3, .5-.4) -- (4.5, -.5-.4);
\draw[thick] (5.3, .5-.4) -- (6.1, -.5-.4);
\draw[thick] (5.3, .5-.4) -- (5.1, -.9-.4);
\draw[thick] (5.3, .5-.4) -- (5.5, -.1-.4);

\fill[black] (4.5+ 1.75+1.1,-.5) circle (0.1);
\fill[black] (5.5+ 1.75+1.1,-.1) circle (0.1);
\fill[black] (5.1+ 1.75+1.1,-.9) circle (0.1);
\fill[black] (6.1+ 1.75+1.1,-.5) circle (0.1);
\draw[thick] (4.5+ 1.75+1.1, -.5) -- (5.5+1.75+1.1, -.1);
\draw[thick] (5.5+ 1.75+1.1, -.1) -- (6.1+1.75+1.1, -.5);
\draw[thick] (6.1+ 1.75+1.1, -.5) -- (5.1+1.75+1.1, -.9);
\draw[thick] (5.1+ 1.75+1.1, -.9) -- (4.5+1.75+1.1, -.5);

\fill[black] (5.3+ 1.75+1.1,.5) circle (0.1);
\draw[thick] (5.3+ 1.75+1.1, .5) -- (4.5+1.75+1.1, -.5);
\draw[thick] (5.3+ 1.75+1.1, .5) -- (6.1+1.75+1.1, -.5);
\draw[thick] (5.3+ 1.75+1.1, .5) -- (5.1+1.75+1.1, -.9);
\draw[thick] (5.3+ 1.75+1.1, .5) -- (5.5+1.75+1.1, -.1);

\fill[black] (5.3+ 1.75+1.1,-1.5) circle (0.1);
\draw[thick] (5.3+ 1.75+1.1, -1.5) -- (4.5+1.75+1.1, -.5);
\draw[thick] (5.3+ 1.75+1.1, -1.5) -- (6.1+1.75+1.1, -.5);
\draw[thick] (5.3+ 1.75+1.1, -1.5) -- (5.1+1.75+1.1, -.9);
\draw[thick] (5.3+ 1.75+1.1, -1.5) -- (5.5+1.75+1.1, -.1);

\end{tikzpicture}
\caption{The four singular graphs.}
\label{Figure: Exceptions to Whitney's Theorem}
\end{figure}
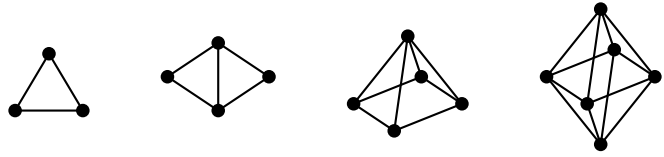

\begin{Theorem}[{Whitney's strong isomorphism theorem}]\label{Whitney}
    If $G$ and $H$ are connected non-exceptional graphs, then every isomorphism $\phi \colon E(G) \to E(H)$ from $L(G)$ to $L(H)$ is induced by an isomorphism $\sigma \colon V(G) \to V(H)$ from $G$ to $H$, that is, if $xy \in E(G)$, then $\phi(xy) = \sigma(x)\sigma(y)$.
\end{Theorem}

    Whitney \cite{whitney1932congruent} proved Theorem~\ref{Whitney} for finite graphs in 1932. A short alternative proof was given by Jung \cite{Jung1966WhitneyStrongIsomorphismTheorem}, who also extended the result to infinite graphs. Jung's paper is in German; for an English version of the proof, see \cite{Hemminger1972whitney} or \cite[Theorem~8.3]{Harary1969GraphTheory}. 
    We shall apply Theorem~\ref{Whitney} in the form of the following corollary: 

\begin{Corollary}[Uniqueness of line graph decompositions]\label{Corollary: nonsingular clique cover is unique}
If $L$ is a connected nonsingular line graph, then $L$ has a unique line graph decomposition \ep{and thus a unique line graph relation}. \end{Corollary}

\begin{proof}
Let $L$ be a connected nonsingular line graph. If $L$ has one vertex, its line graph decomposition is clearly unique. Otherwise, suppose $\mathscr{C}$ and $\mathscr{C}'$ are line graph decompositions of $L$. Let $G$ and $G'$ be the graphs obtained from $\mathscr{C}$ and $\mathscr{C}'$ respectively as in \eqref{Eq: line graph decomposition}. Then $L(G) \cong L$ and $L(G') \cong L$, say by isomorphisms $\phi$ and $\phi'$ respectively. By construction, for all $C \in \mathscr{C}$ and $C' \in \mathscr{C}'$, 
\[
    C \,=\, \set{\phi(\set{C,B}) \,:\, B \in N_G(C)} \quad \text{and} \quad C' \,=\, \set{\phi'(\set{C',B'}) \,:\, B' \in N_{G'}(C')}.
\]
 Since $\psi \defeq (\phi')^{-1} \circ \phi$ is an isomorphism from $L(G)$ to $L(G')$, by Theorem~\ref{Whitney}, $\psi$ is induced by an isomorphism $\sigma$ from $G$ to $G'$, i.e., if $\set{C,C'} \in E(G)$, then $\psi(\set{C,C'}) = \set{\sigma(C), \sigma(C')} \in E(G')$. Now, for any $C \in \mathscr{C}$, we can write
 \begin{align*}
    \sigma(C) \,&=\, \set{\phi'(\set{\sigma(C),B'}) \,:\, B' \in N_{G'}(\sigma(C))} \\ 
    &=\, \set{\phi'(\set{\sigma(C), \sigma(B)}) \,:\, B \in N_G(C)} \\ 
    &=\, \set{\phi' (\psi(\set{C,B})) \,:\, B \in N_G(C)} \,=\, \set{\phi(\set{C,B}) \,:\, B \in N_G(C)}\,=\, C.
 \end{align*}
 Thus $C \in \mathscr{C}'$, 
 and hence $\mathscr{C} \subseteq \mathscr{C}'$. A symmetrical argument shows that $\mathscr{C}' \subseteq \mathscr{C}$ and thus $\mathscr{C} = \mathscr{C}'$. 
 %
%
%
%
\end{proof}


With Corollary~\ref{Corollary: nonsingular clique cover is unique} in hand, it is not difficult to argue that if a Borel graph $L$ is a line graph and all its components are nonsingular, then the unique line graph relation on $L$ must be Borel. On the other hand, if all components of $L$ \emph{are} singular, then in particular they are finite, and it is again straightforward to find a Borel line graph relation on $L$ by picking one of the finitely many such relations on each component of $L$. (This is an instance of the generally well-understood fact that Borel combinatorics essentially trivialize on Borel graphs with finite components, see, e.g., \cites[\S5.3]{Pikhurko2021Survey}[\S2.2]{asiLLL}.) 
The difficult case in the proof of Theorem~\ref{Theorem: A Borel line graph decomposition Exists} is when $L$ has a mixture of singular and nonsingular components. The challenge is that it may be impossible to separate the singular components from the nonsingular ones in a Borel way: the union of all singular components of $L$ is a coanalytic---but not necessarily Borel---set. 
To overcome this difficulty, we use Corollary~\ref{Corollary: nonsingular clique cover is unique} and the \hyperref[Theorem: Analytic Separation]{analytic separation theorem} to first construct a Borel relation $R$ on $E(L)$ that induces a line graph relation on every infinite component of $L$, but can behave arbitrarily on finite components. Next we consider the following two sets:
\begin{align*}
    A_1 \,&\defeq\, \set{x \in V(L) \,:\, \text{the component of $x$ is infinite}},\\
    A_2 \,&\defeq\, \set{x \in V(L) \,:\, \text{$R$ does {not} induce a line graph relation on the component of $x$}}.
\end{align*}
These sets are analytic and---by the construction of $R$---disjoint. With the help of invariant analytic separation (Lemma~\ref{Lemma: invariant analytic separation}), we are able to find a Borel set $B$ such that $A_1 \subseteq B \subseteq A_2^c$ and $B$ is a union of connected components of $L$. Since $B \cap A_2 = \0$, every component of $L$ contained in $B^c$ is finite, which allows us to modify $R$ on $B^c$ to obtain a desired line graph relation on $L$. The details are presented in \S\ref{Section: Proof of Theorem A Borel line graph decomposition Exisits}.

The argument sketched above is representative of the techniques used in this paper, in that it involves a series of applications of analytic separation to construct a Borel structure with desirable combinatorial properties. The proof of Theorem~\ref{Theorem: smoothness implies Borel G} relies on similar ideas, but with even more rounds of analytic separation.




\subsection{Analyzing smooth line graph relations \ep{Theorem~\ref{Theorem: smoothness implies Borel G}}}\label{section: Analyzing smooth line graph relations}



The proof of the implication \ref{item:onesmooth} $\Rightarrow$ \ref{item:Blg} is not as straightforward as may initially appear. To indicate the source of the difficulty, let us sketch an obvious naive approach (which ends up failing). Suppose that $\sim$ is a smooth line graph relation on $L$ and let $f \colon E(L) \to \R$ be a Borel function witnessing the smoothness of $\sim$. This means that for each point $x$ in the image of $f$, $f^{-1}(x)$ is a $\sim$-class. 
Recall the endpoints of the edges of each $\sim$-class induce a clique in $L$; let us denote this clique by $C_x \subseteq V(L)$. For $x \in \R \setminus \im(f)$, set $C_x \defeq \0$. 
In an attempt to mimic \eqref{Eq: line graph decomposition}, let us consider the graph $G$ with $V(G) \defeq \R$ and 
\begin{align*}
        E(G) \,&\defeq\, \{xy \in [\R]^2 \,:\, C_x \cap C_y \neq \0\}, 
\end{align*}
and define a map $\phi \colon E(G) \to V(L)$ by letting $\phi(xy)$ for each edge $xy \in E(G)$ be the (necessarily unique) vertex in $C_x \cap C_y$. Ideally, $\phi$ witnesses $L(G) \cong_B L$ by $\phi$. Unfortunately, there are two issues with the construction:

\begin{itemize}[wide]
    \item First, the map $\phi$ defined in this way is an embedding of $L(G)$ into $L$, but it is only surjective if every vertex of $L$ is incident to exactly two $\sim$-classes. In general, some vertices of $L$ may be incident to one $\sim$-class or be isolated. Note that the set of all isolated vertices, as well as the set of all vertices incident to a single $\sim$-class, need not be Borel. 

    \item The second problem is that the set $E(G)$ defined above is analytic but not necessarily Borel. In other words, $G$ may fail to be a Borel graph.
\end{itemize}

The crux of the difficulty here is that the following relation may not be Borel:
\[
    R \,\defeq\, \set{(v, x) \in V(L) \times \R \,:\, \exists \, e \in E(L) \, (f(e) = x, \, v \in e)}.
\]
To circumvent this obstacle, we repeatedly apply \hyperref[Theorem: Analytic Separation]{analytic separation} to construct a sequence $R_0$, $R_1$, $R_2$, $R_3$, $R_4$, $R_5$ of Borel relations that in some sense ``approximate'' $R$. With care, we are able to ensure that the final relation, $R_5$, has the following properties:
\begin{itemize}
    \item $R \subseteq R_5$, and if $v \, R_5 \, x$ and $x \in \im(f)$, then $v \, R\, x$,
    \item every vertex of $L$ $R_5$-relates to at most two elements of $\R$,
    \item every element of $\R \setminus \im(f)$ $R_5$-relates to at most one vertex of $L$.
\end{itemize}
We then show that these properties enable us to use $R_5$ in place of $R$ in the construction of a Borel graph $G$ with $L(G) \cong_B L$. The details are given in \S\ref{Section: proof of Theorem smoothness implies Borel G}.


\subsection{Finishing the proof}\label{subsec:reduction_loc_countable}

Let $L$ be a Borel graph with a nonsmooth Borel line graph relation $\sim$. To obtain a Borel copy of $\K_0$ in $L$, we seek a Borel induced subgraph $H \subseteq L$ such that:
    \begin{itemize}
        \item every component of $H$ is a clique, and
        \item the equivalence relation $\equiv_H$ on $V(H)$ whose classes are the components of $H$ is nonsmooth.
    \end{itemize}
Once we find such $H$, we can take $\phi \colon \mathcal{C} \to V(H)$ to be a Borel embedding from $\E_0$ to $\equiv_H$ guaranteed by the \hyperref[Theorem: E0 dichotomy]{$\E_0$-dichotomy} and observe that $L[\im(\phi)]$ is a Borel copy of $\K_0$ in $L$, as desired.

To motivate our construction of $H$, note that since $\sim$ is nonsmooth, the $\E_0$-dichotomy yields a Borel embedding $\rho \colon \mathcal{C} \to E(L)$ from $\E_0$ to $\sim$. In particular, if $\alpha$, $\beta \in \mathcal{C}$ are not $\E_0$-related, then $\rho(\alpha) \not\sim \rho(\beta)$. We want to strengthen this property as follows:
\begin{idea}\label{eq:rho}
    \textsl{If $\alpha$, $\beta \in \mathcal{C}$ are not $\E_0$-related, then the endpoints of the edge $\rho(\alpha)$ are not adjacent to the endpoints of $\rho(\beta)$.}
\end{idea}

\noindent It is not hard to see that if \eqref{eq:rho} holds, we can let $H$ be the subgraph of $L$ induced by the vertices incident to $\im(\rho)$. In order to find $\rho$ satisfying \eqref{eq:rho}, we rely on the following lemma:

\begin{Lemma}\label{lemma:hom}
    Let $X$ be a standard Borel space and let $E \subseteq X^2$ be a nonsmooth Borel equivalence relation on $X$. If $R \subseteq X^2$ is a Borel set such that for each $x \in X$, the restriction of $E$ to the set
    \[
        R(x) \,\defeq\, \set{y \in X \,:\, x\,R\, y}
    \]
    is smooth, then there is a Borel injective homomorphism $\rho \colon \mathcal{C} \to X$ from $(\E_0^c,\, \E_0)$ to $(R^c,\, E)$.
\end{Lemma}

Here, given binary relations $(E_1, \ldots, E_n)$ on a set $X$ and $(F_1, \ldots, F_n)$ on a set $Y$, a \emphd{homomorphism} from $(E_1, \ldots, E_n)$ to $(F_1, \ldots, F_n)$ is a function $\rho : X \to Y$ such that $x\, E_i \, y \Rightarrow \rho(x)\,  F_i\, \rho(y)$ for all $x$, $y \in X$ and $1 \leq i \leq n$. We derive Lemma~\ref{lemma:hom} from the $\E_0$-dichotomy and a Mycielski-style theorem due to Miller \cite[Proposition 3]{miller2011classicalProofOfTheKZcanonization}. We then argue that it can be applied with $E= {\sim}$ and
\[
    R \,\defeq\, \set{(e, e') \,:\, \text{$e$ and $e'$ have adjacent endpoints}},
\]
resulting in a mapping $\rho \colon \mathcal{C} \to E(L)$ with the desired properties. The details are presented in \S\ref{Section: reduction to locally countable case}.

\section{Proof of Theorem \ref{Theorem: A Borel line graph decomposition Exists}}\label{Section: Proof of Theorem A Borel line graph decomposition Exisits}

\begin{Theorem*}[\ref{Theorem: A Borel line graph decomposition Exists}]
If $L$ is a Borel graph that is a line graph, then $L$ has a Borel line graph relation.
\end{Theorem*}
\begin{proof}
Given $e$, $f \in E(L)$, we write $e \star f$ whenever there exists a clique $C$ in $L$ with $e$, $f \in E(C)$. Note that the relation $\star$ is Borel, since
\begin{align*}
    \set{x_1, x_2} \star \set{x_3, x_4} \iff \forall\, 1 \leq i, \, j \leq 4 \, (x_i = x_j \text{ or } x_ix_j \in E(L)).
\end{align*}
Let $\sim_L$ be an arbitrary (not necessarily Borel) line graph relation on $L$. Then ${\sim_L} \subseteq {\star}$. Call an induced subgraph $\Gamma \subseteq L$ \emphdef{nice} if $\Gamma$ is connected, finite, and $|V(\Gamma)|\geq 7$. Note that if $\Gamma$ is a nice subgraph of $L$, then it is a nonsingular line graph, as it is an induced subgraph of $L$ and all singular graphs have at most 6 vertices. Thus, by Corollary~\ref{Corollary: nonsingular clique cover is unique}, every nice graph $\Gamma$ has a unique line graph relation, $\sim_{\Gamma}$. By Lemma \ref{lemma: line graph relation reduction}, ${\sim_L}|_\Gamma$ is also a line graph relation on $\Gamma$, and thus ${\sim_\Gamma} = {{\sim_L}|_\Gamma}$.

Define relations $R_1$ and $R_2$ on $E(L)$ as follows:
\begin{align*}
    e \,R_1\, f &\iff \exists \text{ nice $\Gamma \subseteq L$ with $e$, $f \in E(\Gamma)$ and $e \sim_{\Gamma} f$.}\\
    e \,R_2\, f &\iff e \star f \text{ and } \forall \text{ nice $\Gamma \subseteq L$ $\left(e, f \in E(\Gamma) \Rightarrow e \sim_{\Gamma} f\right)$.}
\end{align*}

\begin{Claim}\label{Claim: simone subset of simtwo} The relations $R_1$ and $R_2$ have the following properties:
    \begin{enumerate}[label={\ep{\normalfont\roman*}}]
        \item\label{item:inclusions} $R_1 \subseteq R_2 \subseteq {\star}$,
        \item\label{item:infinite_components} if $H$ is an infinite component of $L$, then $R_1|_H =\, R_2|_H = {\sim_{L}}|_H$,
        \item\label{item:complexity} $R_1$ is analytic, while $R_2$ is coanalytic.
    \end{enumerate}
%
\end{Claim}
\begin{claimproof}
    We start by observing that if edges $e$, $f \in E(L)$ are contained in some nice graph, then
    \begin{equation}\label{eq:three-way-equivalence}
        e \,R_1\, f  \iff  e\,R_2\, f  \iff  e \sim_L f,
    \end{equation}
    because ${\sim_\Gamma} = {{\sim_L}|_\Gamma}$ for every nice graph $\Gamma$ and ${\sim_L} \subseteq {\star}$.

    \smallskip
    
    \ref{item:inclusions} The inclusion $R_2 \subseteq {\star}$ is clear, while $R_1 \subseteq R_2$ follows by \eqref{eq:three-way-equivalence}. 

    \smallskip
    
    \ref{item:infinite_components} 
    If $H$ is an infinite component of $L$, then for any pair of edges $e$, $f \in E(H)$ we can find a nice subgraph of $H$ containing both $e$ and $f$, so $R_1|_H =\, R_2|_H = {\sim_{L}}|_H$ by \eqref{eq:three-way-equivalence}. 

    \smallskip

    \ref{item:complexity} For each $n \in \N$, let
    \begin{align*}
        P_n \,\defeq\, &\big\{(e,f, v_1, \ldots, v_n) \in E(G)^2 \times V(G)^n \,:\, \\
        &\hspace{0.5in}\text{the graph $\Gamma \defeq L[\set{v_1, \ldots, v_n}]$ is nice, $e$, $f \in E(\Gamma)$, and $e \sim_\Gamma f$}\big\}.
    \end{align*}
    The statement ``the graph $\Gamma \defeq L[\set{v_1, \ldots, v_n}]$ is nice, $e$, $f \in E(\Gamma)$, and $e \sim_\Gamma f$'' can be expressed as a Boolean combination of statements of the form ``$v_i = v_j$,'' ``$v_i v_j \in E(L)$,'' ``$e = \set{v_i, v_j}$,'' and ``$f = \set{v_i, v_j}$.'' It follows that the set $P_n$ is Borel. By definition,
\begin{align*}
    e \,R_1\, f \iff \exists\, n \in \N,\, \exists\, (v_1, \ldots, v_n) \in V(G)^n \, \big((e,f,v_1, \ldots, v_n) \in P_n\big),
\end{align*}
which shows that $R_1$ is a countable union of analytic sets, hence it is itself analytic. The proof that $R_2$ is coanalytic is similar, so 
%
we omit the details.
\end{claimproof}

By Claim \ref{Claim: simone subset of simtwo} and the {analytic separation theorem}, there exists a Borel set $R \subseteq E(L)^2$ such that $R_1 \subseteq R \subseteq R_2$. Let $\equiv$ be the equivalence relation on $V(L)$ whose classes are the components of $L$. Note that $\equiv$ is analytic (see Example~\ref{Example: Analytic2}). For a vertex $x \in V(L)$, let $[x]$ denote the component of $L$ containing $x$, and define:
\begin{align*}
    A_1 \,&\defeq\, \{x \in V(L) \,:\, [x] \text{ is infinite}\},\\
    A_2 \,&\defeq\, \{x \in V(L) \,:\, R|_{[x]} \text{ is \textit{not} a line graph relation on $[x]$}\}.
\end{align*}

\begin{Claim}\label{Claim: A1 and A2 analytic}
$A_1$ and $A_2$ are disjoint analytic $\equiv$-invariant sets.
\end{Claim}
\begin{claimproof}
    That $A_1$ and $A_2$ are $\equiv$-invariant is immediate from the way they are defined. Next, we write
    \[
        x \in A_1 \iff \forall\, n \in \N,\, \exists\, y_1, \dots, y_n \in V(L) \, \big(y_1, \, \ldots, \, y_n \text{ are distinct and } \forall i \in [n] \, (x \equiv y_i)\big),
    \]
    which shows that $A_1$ is a countable intersection of analytic sets, so it is itself analytic. 
 To see that $A_2$ is analytic, recall that by Lemma \ref{lemma: equiv relation is line graph decomp relation}, $R|_{[x]}$ is a line graph relation 
 if and only if:
\begin{enumerate}[label=\ep{\normalfont{}\arabic*}]
    \item\label{item:equivalence} $R|_{[x]}$ is an equivalence relation,
    \item\label{item:clique} each equivalence class of $R|_{[x]}$ is the edge set of a clique in $[x]$, and
    \item\label{item:2} for each vertex $y \equiv x$, $y$ is incident to at most two equivalence classes of $R|_{[x]}$. 
\end{enumerate}
For the first condition, we have:
\begin{align*}
    \ref{item:equivalence} \iff \forall\, e, f, g \in E([x])\,
    \bigg(&\, e \, R \, e,\ e \, R \, f \Rightarrow f \, R \, e,\     
    \Big(e \, R \, f \text{ and }f \, R \, g
    \Big)\, \Rightarrow e \, R \, g  \bigg).
\end{align*}
For the second condition, assuming \ref{item:equivalence} holds, we have
\begin{align*}
    \ref{item:clique} \iff \forall\, e = \set{a,b}, f = \set{c,d} \in E([x]) \, \bigg(e \, R \, f \Rightarrow  &\Big(a = c \text{ or } \big(\set{a,c} \in E(L), \, \set{a,c} \,R\, e \big) \Big) \bigg). 
\end{align*}
For the third condition, assuming $\ref{item:equivalence}$ holds, we have
\begin{align*}
    \ref{item:2} \iff \lnot \Bigg( \exists\, y \equiv x,\, \exists\, e, f, g \in E([x])\, \Big(y \in e \cap f\cap g,\, e \, R^c \, f,\, e \, R^c \, g,\, f \, R^c \, g \Big)\Bigg).
\end{align*}
Since $R$ is Borel and the relation
\begin{align*}
    e \in E([x]) \iff \exists\, u \equiv x \, (u \in e),
\end{align*}
is analytic, these three conditions define coanalytic sets, and thus $A_2$ is analytic.

Finally, to see that $A_1$ and $A_2$ are disjoint, let $H$ be an infinite component of $L$. By Claim~\ref{Claim: simone subset of simtwo}, $R_1|_H =\, R_2|_H = {\sim_{L}}|_H$. Since $R_1 \subseteq R \subseteq R_2$, it follows that $R|_H = {\sim_{L}}|_H$. In particular, $R|_H$ is a line graph relation on $H$, so $H$ is contained in $A_2^c$, as desired. 
\end{claimproof}


By Claim \ref{Claim: A1 and A2 analytic}, we may apply {invariant analytic separation} (Lemma~\ref{Lemma: invariant analytic separation}) with $X = V(L)$, $E = {\equiv}$, $Y = A_1$, and $Z = A_2$ to obtain a Borel ${\equiv}$-invariant set $B \subseteq V(L)$ such that $A_1 \subseteq B \subseteq A_2^c$.
Since $B \subseteq A_2^c$, it follows that $R|_{B}$ is a Borel line graph relation on $L[B]$. On the other hand, since $A_1 \subseteq B$, every component of $L[B^c]$ is finite. This means that we may employ the Luzin--Novikov theorem to pick, in a Borel way, a single line graph relation on each component of $L[B^c]$ and form a Borel line graph relation $R^*$ on $L[B^c]$. (Since arguments dealing with Borel graphs with finite components in this manner are standard, we defer the details to Appendix \ref{sec:appendix}.) As $B$ is $\equiv$-invariant, we conclude that $R|_{L[B]} \cup R^*$ is a desired Borel line graph relation on $L$.
\end{proof}

\section{Proof of Theorem \ref{Theorem: smoothness implies Borel G}, \ref{item:onesmooth} $\Rightarrow$ \ref{item:Blg}}\label{Section: proof of Theorem smoothness implies Borel G}

\begin{Theorem*}[\ref{Theorem: smoothness implies Borel G}, \ref{item:onesmooth} $\Rightarrow$ \ref{item:Blg}] 
Let $L$ be a Borel graph that is a line graph. If $L$ has a smooth line graph relation $\sim$, then $L$ is a Borel line graph.
\end{Theorem*}
\begin{proof}
Let $f : E(L) \to \mathbb{R}$ witness the smoothness of $\sim$. Define $R$, $R' \subseteq V(L) \times \mathbb{R}$ by:
\begin{align*}
    v\,R\,x \iff \exists\, uw &\in E(L)\, \big(f(uw) = x \text{ and } (v = u \text{ or } v = w)\big),\\
    v\,R'\,x \iff \forall\, uw &\in E(L) \, \Big(f(uw) = x \Rightarrow \\
    &\hspace{1in}\big((vu \in E(L),\, f(vu) = x) \text{ or } (vw \in E(L),\, f(vw) = x)\big)\Big).
\end{align*}
Note that $R \subseteq R'$ and $R' \setminus R = \{(v,x) 
 \in V(L) \times \R \,:\, x \notin \im(f)\}$.  
Since $f$ is Borel, it follows that $R$ is analytic and $R'$ is coanalytic.

As $R \subseteq R'$, {analytic separation} yields a Borel set $R_0$ such that
$R \subseteq R_0 \subseteq R'$. Define $B_0 \subseteq R_0$ by
\[B_0 \,\defeq\, \{(v,z)\in R_0 \,:\, \exists\, u \in V(L)\, ((u,z) \in R_0,\, uv \in E(L),\, f(uv)\neq z) \}. \]
Since $R_0$ is Borel, $B_0$ is analytic. Note that if $(v,z) \in R$, then $z \in \im(f)$, so if $u$ is a neighbor of $v$ such that $(u,z) \in R_0$, then $(u,z) \in R$ as well. Thus both $u$ and $v$ are incident to edges that are mapped by $f$ to $z$. Since $f^{-1}(z)$ is the edge set of a clique, it follows that $f(uv) = z$ for all such $u$, so $(v,z) \notin B_0$. In other words, $R \cap B_0 = \0$.

By analytic separation, there is a Borel set $R_1$ with
$R \subseteq R_1 \subseteq R_0 \setminus B_0$.
Define $B_1 \subseteq R_1$ by
\[B_1 \,\defeq\, \{(v,z) \in R_1 \,:\, \exists\, u \in V(L)\, (u \neq v,\, uv \notin E(L),\, (u,z) \in R_1)\}.\]
Again, $B_1$ is analytic. Moreover, for each $z \in \im(f)$, the vertices of $L$ to which $z$ is $R$-related form a clique, so $R \cap B_1 = \0$. Thus, by analytic separation, there is a Borel set $R_2$ with
$R \subseteq R_2 \subseteq R_1 \setminus B_1$.

Next we define a subset $B_2 \subseteq R_2$ by
\[B_2 \,\defeq\, \{(v,z)\in R_2 \,:\, \exists\, x,y \in \R\, ((v,x), (v,y) \in R,\, |\{x, y, z\}| = 3 )\}. \]
Since $R_2$ is Borel and $R$ is analytic, $B_2$ is analytic. Furthermore, each vertex of $L$ can be $R$-related to at most two elements of $\R$, so $R \cap B_2 = \0$. By analytic separation, there exists a Borel set $R_3$ with
$R \subseteq R_3 \subseteq R_2\setminus B_2$. Let $B_3 \subseteq R_3$ be defined as follows:
\[B_3 \,\defeq\, \{(v,z)\in R_3 \,:\, \exists\, x,y \in \R\, ((v,x) \in R,\, (v,y) \in R_3,\, |\{x, y, z\}| = 3 )\}. \]
Since $R_3$ is Borel and $R$ is analytic, $B_3$ is analytic. Take any $(v,z) \in B_3$ with $(v,x) \in R$, $(v,y) \in R_3$, and $|\{x, y, z\}| = 3$. 
If $(v,z) \in R$, then $(v,y) \in B_2$, and thus $(v,y) \notin R_3$, which is a contradiction. Thus, $R \cap B_3 = \0$. By analytic separation, there is a Borel set $R_4$ with $R \subseteq R_4 \subseteq R_3\setminus B_3$. Let
\[B_4 \,\defeq\, \{(v,z)\in R_4 \,:\, \exists\, x,y \in \R\, ((v,x) \in R_4,\, (v,y) \in R_4,\, |\{x, y, z\}| = 3 )\}. \]
Since $R_4$ is Borel, $B_4$ is analytic. If $(v,z) \in B_4$ with $(v,x) \in R_4$, $(v,y) \in R_4$, and $|\{x, y, z\}| = 3$, and $(v,z) \in R$, then $(v,x) \in B_3$, which is impossible. Therefore, $R \cap B_4 = \0$.

Finally, analytic separation yields a Borel set $R_5$ with $R \subseteq R_5 \subseteq R_4\setminus B_4$. Observe that $R_5$ has the following properties:
\begin{enumerate}[label=\ep{\normalfont\roman*}]
    \item\label{R_5 property 1} $R \subseteq R_5 \subseteq R'$.
\end{enumerate}
This is clear from the construction.
\begin{enumerate}[label=\ep{\normalfont\roman*},resume]
    \item\label{R_5 property 2} Every vertex of $L$ $R_5$-relates to at most two elements of $\R$.
\end{enumerate}
This follows since $R_5 \cap B_4 = \0$.
\begin{enumerate}[label=\ep{\normalfont\roman*},resume]
    \item\label{R_5 property 3} Every element of $\R \setminus \im(f)$ $R_5$-relates to at most one vertex of $L$.
\end{enumerate}
Indeed, suppose $z \in \R$ $R_5$-relates to two different vertices $u$, $v \in V(L)$. Since $R_5 \cap B_1 = \0$, it follows that $uv \in E(L)$. Then, since $R_5 \cap B_0 = \0$, we have $f(uv) = z$, i.e., $z \in \im(f)$, as claimed.

\begin{figure}[t]
\begin{tikzpicture}[yscale=1.1]


\node[circle,fill=black,inner sep=0pt,minimum size=4pt] (v1) at (-.3,-.2) {};
\node[circle,fill=black,inner sep=0pt,minimum size=4pt] (v2) at (0.3,0.2) {};
\node[circle,fill=black,inner sep=0pt,minimum size=4pt] (v3) at (.6,-0.2) {};
\node[circle,fill=black,inner sep=0pt,minimum size=4pt] (v4) at (1.2,-0.5) {};
\node[circle,fill=black,inner sep=0pt,minimum size=4pt] (v5) at (2,0.2) {};
\node[circle,fill=black,inner sep=0pt,minimum size=4pt] (v6) at (2.7,-.2) {};

\node[circle,fill=black,inner sep=0pt,minimum size=4pt] (v7) at (3.5,0) {};
\node[circle,fill=black,inner sep=0pt,minimum size=4pt] (v8) at (4.2,0) {};
\node[circle,fill=black,inner sep=0pt,minimum size=4pt] (v9) at (5,0) {};

\node[circle,fill=black,inner sep=0pt,minimum size=4pt] (x1) at (0,-1.5) {};
\node[circle,fill=black,inner sep=0pt,minimum size=4pt] (x2) at (1,-1.5) {};
\node[circle,fill=black,inner sep=0pt,minimum size=4pt] (x3) at (2,-1.5) {};

\node[circle,fill=black,inner sep=0pt,minimum size=4pt] (x5) at (4,-1.5) {};
\node[circle,fill=black,inner sep=0pt,minimum size=4pt] (x6) at (5,-1.5) {};
\node[circle,fill=black,inner sep=0pt,minimum size=4pt] (x7) at (6,-1.5) {};
\node[circle,fill=black,inner sep=0pt,minimum size=4pt] (x4) at (3,-1.5) {};

\node[anchor=south,scale=0.8] at (v1) {$v_1$};
\node[anchor=south,scale=0.8] at (v2) {$v_2$};
\node[anchor=south,scale=0.8] at (v3) {$v_3$};
\node[anchor=south,scale=0.8] at (v4) {$v_4$};
\node[anchor=south,scale=0.8] at (v5) {$v_5$};
\node[anchor=south,scale=0.8] at (v6) {$v_6$};
\node[anchor=south,scale=0.8] at (v7) {$v_7$};
\node[anchor=south,scale=0.8] at (v8) {$v_8$};
\node[anchor=south,scale=0.8] at (v9) {$v_9$};

\node[anchor=north,scale=0.8] at (x1) {$x_1$};
\node[anchor=north,scale=0.8] at (x2) {$x_2$};
\node[anchor=north,scale=0.8] at (x3) {$x_3$};
\node[anchor=north,scale=0.8] at (x4) {$x_4$};
\node[anchor=north,scale=0.8] at (x5) {$x_5$};
\node[anchor=north,scale=0.8] at (x6) {$x_6$};
\node[anchor=north,scale=0.8] at (x7) {$x_7$};

\draw[very thick] (v1) -- (v2) -- (v3) -- (v1) (v3) -- (v4) -- (v5) -- (v3) (v5) -- (v6);

\draw[dashed,thick] (v1) -- (x1) (v2) -- (x1) (v3) -- (x1) (v3) -- (x2) (v4) -- (x2) (v5) -- (x2) (v5) -- (x3) (v6) -- (x3);

\draw[dotted,thick] (v6) -- (x4) (v7) -- (x5) (v7) -- (x6) (v9) -- (x7);


\draw [decorate, 
    decoration = {brace,
        raise=3pt, amplitude=5pt,
        aspect=.5}] (2.2,-1.8) --  (-.2,-1.8)
node[pos=0.5,below=5pt,black]{$\im(f)$};



\draw[rounded corners] (-0.6, -.7) rectangle (6, .6) {};
\node at (-1,-0.05) {$L$};

\draw[rounded corners] (-0.6, -1.85) rectangle (6.5, -1.3) {};
\node at (-1,-1.575) {$\R$};



\begin{scope}[xshift=8.5cm]
    \node[circle,fill=black,inner sep=0pt,minimum size=4pt] (v1) at (-.3,-.2) {};
\node[circle,fill=black,inner sep=0pt,minimum size=4pt] (v2) at (0.3,0.2) {};
\node[circle,fill=black,inner sep=0pt,minimum size=4pt] (v4) at (1.2,-0.5) {};

\node[circle,fill=black,inner sep=0pt,minimum size=4pt] (0v8) at (3,-0.2) {};
\node[circle,fill=black,inner sep=0pt,minimum size=4pt] (1v8) at (4,0) {};
\node[circle,fill=black,inner sep=0pt,minimum size=4pt] (v9) at (5,0) {};

\node[circle,fill=black,inner sep=0pt,minimum size=4pt] (x1) at (0,-1.5) {};
\node[circle,fill=black,inner sep=0pt,minimum size=4pt] (x2) at (1,-1.5) {};
\node[circle,fill=black,inner sep=0pt,minimum size=4pt] (x3) at (2,-1.5) {};

\node[circle,fill=black,inner sep=0pt,minimum size=4pt] (x5) at (4,-1.5) {};
\node[circle,fill=black,inner sep=0pt,minimum size=4pt] (x6) at (5,-1.5) {};
\node[circle,fill=black,inner sep=0pt,minimum size=4pt] (x7) at (6,-1.5) {};
\node[circle,fill=black,inner sep=0pt,minimum size=4pt] (x4) at (3,-1.5) {};

\node[anchor=south,scale=0.8] at (v1) {$v_1$};
\node[anchor=south,scale=0.8] at (v2) {$v_2$};
\node[anchor=south,scale=0.8] at (v4) {$v_4$};
\node[anchor=south,scale=0.8] at (0v8) {$(0,v_8)$};
\node[anchor=south,scale=0.8] at (1v8) {$(1,v_8)$};
\node[anchor=south,scale=0.8] at (v9) {$v_9$};

\node[anchor=north,scale=0.8] at (x1) {$x_1$};
\node[anchor=north,scale=0.8] at (x2) {$x_2$};
\node[anchor=north,scale=0.8] at (x3) {$x_3$};
\node[anchor=north,scale=0.8] at (x4) {$x_4$};
\node[anchor=north,scale=0.8] at (x5) {$x_5$};
\node[anchor=north,scale=0.8] at (x6) {$x_6$};
\node[anchor=north,scale=0.8] at (x7) {$x_7$};

\draw[very thick] (v1) -- (x1) (v2) -- (x1) (x1) -- (x2) (v4) -- (x2) (x2) -- (x3) (x3) -- (x4) (x5) -- (x6) (v9) -- (x7) (0v8) -- (1v8);

\node at (3, -2.2) {$G$};
\end{scope}

\node[scale=3] at (7, -0.8) {$\rightsquigarrow$};

\end{tikzpicture}
\caption{The construction of $G$ using the relation $R_5$. Here the dashed edges represent the relation $R$ and the dotted ones represent the relation $R_5 \setminus R$. In this example, $V_0 = \set{v_8}$, $V_1 = \set{v_1, v_2, v_4, v_9}$, and $V_2 = \set{v_3, v_5, v_6, v_7}$.}
\label{Figure: L to G}
\end{figure}
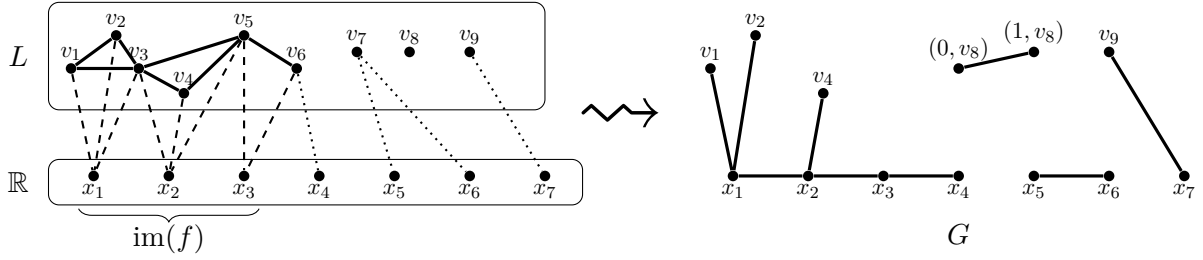

Having found a Borel relation $R_5$ satisfying conditions \ref{R_5 property 1}--\ref{R_5 property 3}, we can now define a Borel graph $G$ such that $L(G) \cong_B L$. To this end, let
\begin{align*}
V_0 \,&\defeq\, \{v \in V(L) \,:\, \forall\, x \in \R\, ((v,x) \notin R_5)\}, \\
V_1 \,&\defeq\, V(L) \setminus (V_0 \cup V_2), \text{ where}\\
V_2 \,&\defeq\, \{v \in V(L) \,:\, \exists\, x, y \in \R\, ((v,x), (v,y) \in R_5,\, x \neq y)\}.
\end{align*}
As $R_5$ is Borel, the {Luzin--Novikov theorem} together with property \ref{R_5 property 2} of $R_5$ shows that $V_0$ and $V_2$ are Borel sets, and thus $V_1$ is Borel as well. Without loss of generality (e.g., by replacing $V(L)$ with $\set{2} \times V(L)$), we may assume that the sets $\R$, $V(L)$, and $\set{0,1}\times V(L)$ are disjoint. We then construct $G$ as follows:
\[
    V(G) \,\defeq\, (\set{0,1} \times V_0) \cup V_1 \cup \R, \quad\quad E(G) \,\defeq\, E_0 \cup E_1 \cup E_2,
\]
where
\begin{align*}
    E_0\,&\defeq\, \Big\{\{(0,v), (1,v)\} \,:\, v \in V_0\Big\},\\
    E_1 \,&\defeq\, \Big\{\{v,x\} \,:\, v \in V_1,\, (v,x) \in R_5 \Big\},\\
    E_2 \,&\defeq\, \Big\{\{x,y\} \,:\, x \neq y \text{ and } \exists\, v \in V_2\, \big((v,x), (v,y) \in R_5\big)\Big\}.
\end{align*}
This construction is illustrated in Fig.~\ref{Figure: L to G}. To see that $G$ is a Borel graph, we need to verify that the sets $E_0$, $E_1$, and $E_2$ are Borel. For $E_0$ and $E_1$, this is clear. For $E_2$, notice that if $x$, $y \in \im(f)$ are distinct, then there is at most one vertex $v \in V_2$ such that $(v,x)$, $(v,y) \in R_5$, namely the common vertex of the cliques $f^{-1}(x)$ and $f^{-1}(y)$. On the other hand, if, say, $x \in \R \setminus \im(f)$, then by \ref{R_5 property 3}, there is at most one vertex $v$ such that $(v,x) \in R_5$. In either case, there is at most one vertex $v$ with $(v,x)$, $(v,y) \in R_5$ and hence, by the Luzin--Novikov theorem, the set $E_2$ is Borel. 

To argue that $L \cong_B L(G)$, we define a Borel isomorphism $\phi$ from $L(G)$ to $L$ as follows. For $\{(0,v), (1, v)\} \in E_0$, let
$
\phi(\{(0,v), (1, v)\}) \defeq v \in V_0
$,
for $\{v,x\} \in E_1$, define
$\phi(\{v,x\}) \defeq v \in V_1$,
and for $\{x,y\} \in E_2$, let $\phi(\set{x,y})$ be the unique $v \in V_2$ such that $(v,x)$, $(v,y) \in R_5$. 
It is immediate from the definition that $\phi$ is indeed a desired Borel isomorphism.
\end{proof}

\section{Finishing the proof}\label{Section: reduction to locally countable case}

Recall that a subset of a topological space is \emphd{meager} if it is a union of countably many nowhere dense sets. We shall use the following result of Miller \cite{miller2011classicalProofOfTheKZcanonization}:

\begin{Theorem}[{Miller \cite[Proposition 3]{miller2011classicalProofOfTheKZcanonization}}]\label{Theorem:meager}
    Let $R \subseteq \mathcal{C}^2$ be a meager set. Then there exists a continuous injective homomorphism $\rho : \mathcal{C} \to \mathcal{C}$ from $(\E_0^c,\, \E_0)$ to $(R^c,\, \E_0)$. 
\end{Theorem}

With Theorem~\ref{Theorem:meager} in hand, we can prove Lemma~\ref{lemma:hom}:

\begin{Lemma*}[\ref{lemma:hom}] 
Let $X$ be a standard Borel space and let $E \subseteq X^2$ be a nonsmooth Borel equivalence relation on $X$. If $R \subseteq X^2$ is a Borel set such that for each $x \in X$, the restriction of $E$ to the set
    \[
        R(x) \,\defeq\, \set{y \in X \,:\, x\,R\, y}
    \]
    is smooth, then there is a Borel injective homomorphism $\rho \colon \mathcal{C} \to X$ from $(\E_0^c,\, \E_0)$ to $(R^c,\, E)$.
\end{Lemma*}
\begin{proof}
    By the $\E_0$-dichotomy, there is a Borel embedding $f \colon \mathcal{C} \to X$ from $\E_0$ to $E$. Since $f$ is injective, its image is Borel, so we may, without loss of generality, replace $X$ by $\im(f)$ and assume that $f$ is a bijection. (For each $x \in X$, the restriction of $E$ to $R(x) \cap \im(f) \subseteq R(x)$ remains smooth, so the assumptions of the lemma are still satisfied.) Since Borel bijections between standard Borel spaces are isomorphisms \cite[Corollary 15.2]{kechris2012classicalDescriptiveSetTheory}, we may in fact assume that $X = \mathcal{C}$, $f$ is the identity map, and $E = \E_0$. If $A \subseteq \mathcal{C}$ is a Borel set such that $\E_0|_A$ is smooth, then $A$ is meager \cite[Corollary 4.12]{SmoothE0}, so $R(x)$ is meager for all $x \in \mathcal{C}$. By the Kuratowski--Ulam theorem \cite[Theorem 8.41]{kechris2012classicalDescriptiveSetTheory}, it follows that $R$ is a meager subset of $\mathcal{C}^2$. Therefore, we may apply Theorem~\ref{Theorem:meager} to get a continuous (hence Borel) injective homomorphism $\rho : \mathcal{C} \to \mathcal{C}$ from $(\E_0^c,\, \E_0)$ to $(R^c,\, \E_0)$, as desired.
\end{proof}

Now we have all the necessary ingredients to complete the proof of our main result.

\begin{Theorem*}[\ref{Theorem: K0 Dichotomy}]
    Let $L$ be a Borel graph that is a line graph. Then $L$ is a Borel line graph if and only if $L$ does not contain a Borel copy of $\mathbb{K}_0$.
\end{Theorem*}
\begin{proof}
    By Theorem~\ref{Theorem: A Borel line graph decomposition Exists}, there is a Borel line graph relation $\sim$ on $L$. If $\sim$ is smooth, then $L$ is a Borel line graph by Theorem~\ref{Theorem: smoothness implies Borel G}, and thus contains no Borel copies of $\K_0$. Now suppose that $\sim$ is nonsmooth. Our goal is to show that $L$ contains a Borel copy of $\K_0$.

    As $\sim$ is Borel, every $\sim$-class is a Borel subset of $E(L)$. For a $\sim$-class $C$, let $V(C) \subseteq V(L)$ be the set of all vertices incident to an edge in $C$. The set $V(C)$ is Borel since, fixing an arbitrary edge $e \in C$ and a vertex $x \in e$, we can write
    \[
        V(C) \,=\, \set{w \in V(L) \,:\, w = x \text{ or } wx \sim e}.
    \]
    For an edge $e \in E(L)$, let $C_{e}$ be the $\sim$-class containing $e$, and for each $\sim$-class $C$, define
    \[
        S(C) \,\defeq\, \set{e \in E(L) \,:\, e \not\in C \text{ and } V(C_{e}) \cap V(C) \neq \0}.
    \]
    \begin{Claim}\label{Claim:smooth1}
        For each $\sim$-class $C$, $S(C)$ is a $\sim$-invariant Borel set and the relation ${\sim}|_{S(C)}$ is smooth.
    \end{Claim}
    \begin{claimproof}
        Fix a $\sim$-class $C$. It is clear from the definition that $S(C)$ is $\sim$-invariant. Observe that
        \[
            S(C) \,=\, \set{uv \in E(L) \setminus C \,:\, \exists\, w \in V(C) \, (w = u \text{ or } wu \sim uv)}.
        \]
        Since for $e \not \in C$, there can be at most one vertex $w$ in $V(C_e) \cap V(C)$, the set $S(C)$ is Borel by the {Luzin--Novikov theorem}. Additionally, the following function $f \colon S(C) \to V(C)$ is Borel:
    \[
        f(e) \,\defeq\, \text{the unique vertex } w \in V(C_e) \cap V(C).
    \]
    If $e$, $e' \in S(C)$ are $\sim$-equivalent, then $f(e) = f(e')$ by construction. Conversely, if $f(e) = f(e') \eqqcolon w$, then $e$ and $e'$ belong to the same $\sim$-class, namely the unique $\sim$-class other than $C$ incident to $w$. In other words, $f$ witnesses the smoothness of ${\sim}|_{S(C)}$, as desired.
    \end{claimproof}

    Define a relation $R \subseteq E(L)^2$ as follows:
    \[
        x_1x_2 \,R\, y_1y_2 \iff \exists\, i,j \in \set{1,2} \,(x_i y_i \in E(L)).
    \]
    \begin{Claim}\label{Claim:smooth_fibers}
        For each $xy \in E(L)$, the restriction of $\sim$ to $R(xy) \defeq \set{e \in E(L) \,:\, xy \,R\,e}$ is smooth. 
    \end{Claim}
\begin{claimproof}
    Fix an edge $xy \in E(L)$ and observe that
    \begin{equation}\label{eq:Rxy}
        R(xy) \,\subseteq\, C_{xy} \,\cup\, S(C_{xy}) \,\cup\, S(C_x) \,\cup\, S(C_y), 
    \end{equation}
    where $C_x$ and $C_y$ are the $\sim$-classes distinct from $C_{xy}$ containing $x$ and $y$ respectively (if $x$ or $y$ is incident to only one $\sim$-class, we let the corresponding set in \eqref{eq:Rxy} be empty). For each $t \in \set{xy, x, y}$, let $f_t \colon S(C_t) \to \R$ witness the smoothness of ${\sim}|_{S(C_t)}$ (such functions $f_t$ exist by Claim~\ref{Claim:smooth1}). Then the following map $f \colon R(xy) \to \set{0,1,2,3} \times \R$ witnesses the smoothness of ${\sim}|_{R(xy)}$:
    \[
        f(e) \,\defeq\, \begin{cases}
            (0,0) &\text{if } e \in C_{xy},\\
            (1,f_{xy}(e)) &\text{if } e \in S(C_{xy}),\\
            (2, f_{x}(e)) &\text{if } e \in S(C_x) \setminus (C_{xy} \cup S(C_{xy})),\\
            (3, f_{y}(e)) &\text{if } e \in S(C_y) \setminus (C_{xy} \cup S(C_{xy}) \cup S(C_x)).
        \end{cases}\qedhere
    \]
    \end{claimproof}

    With Claim~\ref{Claim:smooth_fibers} in hand, we may apply Lemma~\ref{lemma:hom} to obtain a Borel injective homomorphism $\rho \colon \mathcal{C} \to E(L)$ from $(\E_0^c,\, \E_0)$ to $(R^c,\, \sim)$. Since $\rho$ is injective, its image $\im(\rho)$ is a Borel subset of $E(L)$. Let $U \subseteq V(L)$ be the set of all vertices incident to an edge in $\im(\rho)$. Since every $\E_0$-class is countable, each $\sim$-class contains countably many edges in $\im(\rho)$. As every vertex belongs to at most two $\sim$-classes, it is incident to countably many edges in $\im(\rho)$. It follows by the Luzin--Novikov theorem that the set $U$ is Borel, and hence $H \defeq L[U]$ is a Borel induced subgraph of $L$.
    
    Observe that if $xy$, $yz \in E(H)$, then $xy \sim yz$. Indeed, let $\alpha$, $\beta$, $\gamma \in \mathcal{C}$ be such that the edges $\rho(\alpha)$, $\rho(\beta)$, and $\rho(\gamma)$ are incident to $x$, $y$, and $z$ respectively. Then $\rho(\alpha) \,R\, \rho(\beta) \,R\, \rho(\gamma)$, and thus $\alpha \,\E_0\, \beta \, \E_0 \, \gamma$ because $\rho$ is a homomorphism from $\E_0^c$ to $R^c$. Since $\rho$ is a homomorphism from $\E_0$ to $\sim$, we conclude that $\rho(\alpha) \sim \rho(\beta) \sim \rho(\gamma)$. Therefore, $x$, $y$, and $z$ are all incident to the same $\sim$-class, and thus $xy \sim yz$, as desired.

    We conclude that the edge set of every component of $H$ is contained in a single $\sim$-class. Define the relation ${\equiv_H}$ on $V(H)$ by 
\[x \equiv_H y \iff  \text{ $x$ and $y$ are in the same component of $H$}.\]
Since $H$ is locally countable, ${\equiv_H}$ is Borel by the Luzin--Novikov theorem (see~Examples~\ref{Example: Analytic2} and \ref{Example: Luzin--Novikov}). 

\begin{Claim}
    ${\equiv_H}$ is nonsmooth.
\end{Claim} 
\begin{claimproof}
    Suppose to the contrary that ${\equiv_H}$ is smooth, and let $f : V(H) \to 
\R$ witness the smoothness of ${\equiv_H}$. Then $\xi : E(H) \to \R$ defined by $\xi(xy) \defeq f(x)$ is well-defined and witnesses the smoothness of ${\sim}|_H$. But ${\sim}|_H$ is nonsmooth as $\rho$ is an embedding from $\E_0$ to ${\sim}|_H$.
\end{claimproof}

Since ${\equiv_H}$ is nonsmooth, by the $\E_0$-dichotomy there exists a Borel embedding $\phi : \mathcal{C} \to V(H)$ from $\E_0$ to ${\equiv_H}$. As $H$ is a union of disjoint cliques, it follows that $\phi$ is a Borel isomorphism from $\K_0$ to $L[\im(\phi)]$. Therefore, $L$ contains a Borel copy of $\K_0$, as desired. 
\end{proof}

\subsection*{Acknowledgments}

We thank the anonymous referee for their careful reading of this manuscript and helpful suggestions.

\subsection*{Funding}

This research is partially supported by the NSF grant DMS-2045412 and the NSF CAREER grant DMS-2239187.

 \printbibliography

\appendix

\section{Proof of Theorem \ref{Theorem: smoothness implies Borel G}, \ref{item:onesmooth} $\Rightarrow$ \ref{item:allsmooth}}\label{appendix one smooth implies all smooth}

As Corollary \ref{Corollary: nonsingular clique cover is unique} implies two line graph relations differ only on singular components, which are all finite, the statement \ref{item:onesmooth} $\Rightarrow$ \ref{item:allsmooth} is an immediate corollary of the following lemma:

\begin{Lemma}\label{Lemma: E and E'}
    If $E$ and $E'$ are Borel equivalence relations on a standard Borel space $X$ such that $E$ is smooth and every infinite $E'$-class is also an $E$-class, 
    then $E'$ is smooth. 
\end{Lemma}
\begin{proof}
Let $f : X \to \R$ witness the smoothness of $E$. Define the following subsets of $X$:
\begin{align*}
    A_1 \,\defeq\, &\{x \in X \,:\, [x]_{E'} \text{ is infinite}\},\\
    A_2 \,\defeq\,& \{x \in X \,:\, \forall\, y \in X\, (x E' y\iff x E y) \}.
\end{align*}
Clearly $A_1$ is analytic and $A_2$ is coanalytic. Furthermore, if $x \in A_1$, then $[x]_{E'}  = [x]_{E}$, and thus $x \in A_2$. So $A_1 \subseteq A_2$. Since $A_1$ is $E'$-invariant and analytic, while $A_2$ is coanalytic, invariant analytic separation (Lemma \ref{Lemma: invariant analytic separation}) yields an $E'$-invariant Borel set $B$ such that $A_1 \subseteq B \subseteq A_2$.

Fix a Borel linear ordering, say $\preccurlyeq$, on $X$ (for instance, we may assume that $X = \R$ \cite[Theorem 15.6]{kechris2012classicalDescriptiveSetTheory} and use the standard ordering on $\R$). 
If $x \in B^c$, then $[x]_{E'}$ is finite. Thus there exists a $\preccurlyeq$-minimum element, say $\mu(x) \in B^c$, such that $x \,E'\, \mu(x)$. For each $x \in B^c$, there are only finitely many $y \in X$ such that $x \,E'\, y$, and so the map $\mu : B^c \to B^c$ is Borel by the Luzin--Novikov theorem. Furthermore, for $x$, $y \in B^c$, we have $x \, E'\, y$ if and only if $\mu(x) = \mu(y)$. On the other hand, if $x$, $y \in B$, then, since $B \subseteq A_2$, we have $x \,E'\, y$ if and only if $f(x) = f(y)$.

Without loss of generality, we may assume that $X \cap \R = \0$. Define $g : X\to X \cup \R$ by
\[g(x) \,\defeq\, 
    \begin{cases}
    f(x) &\text{if } x \in B,\\
    \mu(x) &\text{if } x \in B^c.
    \end{cases}
\]
Since $f$ and $\mu$ are Borel functions and $B$ is a Borel set, $g$ is a Borel function. Furthermore, the above discussion implies that $g$ witnesses the smoothness of $E'$, as desired.
\end{proof}

\section{Borel line graph relations for graphs with finite components}\label{sec:appendix}

In this appendix we prove the following lemma, which was used in the proof of Theorem~\ref{Theorem: A Borel line graph decomposition Exists}:

\begin{Lemma}\label{lemma:finite_components}
    Let $L$ be a Borel graph with finite components. If $L$ is a line graph, then $L$ has a Borel line graph relation.
\end{Lemma}

Statements such as Lemma~\ref{lemma:finite_components} are considered routine in descriptive set theory. Indeed, Lemma~\ref{lemma:finite_components} can be seen as a special case of certain general facts about Borel combinatorics on Borel graphs with finite components, for example, \cites[\S5.3]{Pikhurko2021Survey}[\S2.2]{asiLLL}. Nevertheless, in an effort to make this paper more accessible to non-experts, we present a complete proof here. In the following argument, it will be useful to keep in mind that if $X$ is a standard Borel space and $Y$ is a countable set, then a map $f \colon X \to Y$ is Borel if and only if $f^{-1}(y)$ is a Borel set for each $y \in Y$.

\begin{proof}
    Let $L$ be a Borel graph with finite components such that $L$ is a line graph. For each $x \in V(L)$, let $[x]$ denote the component of $L$ containing $x$. Since the components of $L$ are finite, the {Luzin--Novikov theorem} implies that the relation ${\equiv} \defeq \set{(x,y) \in V(L)^2 \,:\, y \in V([x])}$ is Borel (see Examples~\ref{Example: Analytic2} and \ref{Example: Luzin--Novikov}). Fix a Borel linear ordering, say $\preccurlyeq$, on $V(L)$ (for instance, we may assume that $V(L) = \R$ \cite[Theorem 15.6]{kechris2012classicalDescriptiveSetTheory} and use the standard ordering on $\R$). Define a function $r : V(L) \to \N$ by
\begin{align*}
    r(x) = k &\iff \text{ $x$ is the $k$-th element of $V([x])$ under $\preccurlyeq$} \\
    & \iff \exists\, x_1, \dots, x_{k-1} \in V([x])\, \big((x_1 \prec x_2 \prec \cdots \prec x_{k-1} \prec x) \text{ and}\\
    &\hspace{2.5in}\forall\, y \in V([x]){\setminus}\{x_1, \dots, x_{k-1}\}\, (x \preccurlyeq y)\big).
\end{align*}
As $[x]$ is finite, all the quantifiers in the above definition range over finite sets, so the function $r$ is Borel by the Luzin--Novikov theorem. The map $s(x) \defeq |V([x])|$ is also Borel, since we can write
\[
    s(x) = k \iff \big(\exists\, y \equiv x\, (r(y) = k)\big) \text{ and } \big(\forall\, y \equiv x\, (r(y) \leq k)\big).
\]
Next we define, for each positive integer $k$, a partial mapping $n_k \colon V(L) \pto V(L)$ as follows:
\[
    n_k(x) = y \iff y \equiv x \text{ and } r(y) = k.
\]
The function $n_k$ is again Borel. Note that for each $x \in V(L)$,
\[
    V([x]) \,=\, \set{n_1(x), n_2(x), \ldots, n_{s(x)}(x)} \quad \text{and} \quad n_1(x) \prec n_2(x) \prec \cdots \prec n_{s(x)}(x).
\]

Let $\mathcal{G}^{<\infty}$ be the (countable) set of all finite line graphs with vertex set a subset of $\N$. For every $\Gamma \in \mathcal{G}^{<\infty}$, fix an arbitrary line graph relation $\sim_\Gamma$ on $\Gamma$. Given $x \in V(L)$, define $\Gamma_x \in \mathcal{G}^{<\infty}$ by:
\[
    V(\Gamma_x) \,\defeq\, \set{1, 2, \ldots, s(x)}, \quad \quad E(\Gamma_x) \,\defeq\, \big\{\{i, j\} : \set{n_i(x), n_j(x)} \in E(L) \big\}.
\]
Then $r$ establishes an isomorphism $[x] \cong \Gamma_x$, and if $y \equiv x$, then $\Gamma_x = \Gamma_y$. Since the set $E(L)$ and the functions $s$ and $n_k$ for all $k$ are Borel, the map $V(L) \to \mathcal{G}^{<\infty} \colon x \mapsto \Gamma_x$ is Borel as well. (To clarify, this means that for each graph $\Gamma \in \mathcal{G}^{<\infty}$, the set of all $x \in V(L)$ with $\Gamma_x = \Gamma$ is Borel.) 
%

Finally, we define a relation $\sim$ on $E(L)$ as follows: if $e = xy$ and $f = uv$, let
\[e \sim f \iff x \equiv u \text{ and } r(x)r(y) \sim_{\Gamma_x} r(u)r(v).\]
In other words, $\sim$ is obtained by ``copying'' $\sim_{\Gamma_x}$ from $\Gamma_x$ onto $[x]$ for each $x \in V(L)$ in the obvious way. It is now clear that $\sim$ is a desired Borel line graph relation on $L$.
%
\end{proof}

\end{document}